\documentclass[reqno, 12pt]{amsart}

\title{Clique factors in randomly perturbed graphs: the transition points}
\author{ Sylwia Antoniuk} 
    \address{Department of Discrete Mathematics, Faculty of Mathematics and CS, Adam Mickiewicz University,
Poznań, Poland}
    \email{sylwia.antoniuk@amu.edu.pl}
 \author {Nina Kam\v{c}ev}
\address{Department of Mathematics, Faculty of Science, University of Zagreb, 10000 Zagreb, Croatia}
\email{nina.kamcev@math.hr}
\author{
	Christian Reiher}
        \address{Fachbereich Mathematik, Universität Hamburg, Hamburg, Germany}
    \email{christian.reiher@uni-hamburg.de}

\makeatletter
\let\origsection=\section \def\section{\@ifstar{\origsection*}{\mysection}} 
\def\mysection{\@startsection{section}{1}\z@{.7\linespacing\@plus\linespacing}{.5\linespacing}{\normalfont\scshape\centering\S\hspace{1pt}}}
\makeatother        

\usepackage{amsmath,amssymb,amsthm}
\usepackage{mathrsfs}
\usepackage{mathabx}\changenotsign
\usepackage{dsfont}
\usepackage{url}

\usepackage{xcolor}
\usepackage[backref]{hyperref}
\hypersetup{
colorlinks,
linkcolor={red!60!black},
citecolor={green!60!black},
urlcolor={blue!60!black}
}

\usepackage[open,openlevel=2,atend]{bookmark}

\usepackage{doi}

\usepackage{natbib}
    \def\newblock{\ }%
    \bibpunct[, ]{[}{]}{,}{n}{}{,}%

\usepackage{hyperref}
\usepackage[T1]{fontenc}
\usepackage{lmodern}
\usepackage[babel]{microtype}
\usepackage[english]{babel}

\usepackage{tikz}
\usetikzlibrary{calc,decorations.pathmorphing}
\pgfdeclarelayer{background}
\pgfdeclarelayer{foreground}
\pgfdeclarelayer{front}
\pgfsetlayers{background,main,foreground,front}

\usepackage{geometry}
\numberwithin{equation}{section}
\numberwithin{figure}{section}

\usepackage{enumitem}

\theoremstyle{plain}
\newtheorem{thm}{Theorem}[section]

\newtheorem{claim}[thm]{Claim}

\newtheorem{lemma}[thm]{Lemma}

\newtheorem{remark}[thm]{Remark}

\theoremstyle{definition}

\let\eps=\varepsilon
\let\theta=\vartheta
\let\rho=\varrho
\let\phi=\varphi

\def\cA{{\mathcal A}}
\def\cB{{\mathcal B}}
\def\cC{{\mathcal C}}

\def\cF{{\mathcal F}}
\def\cH{{\mathcal H}}
\def\cK{{\mathcal K}}

\def\cM{{\mathcal M}}
\def\cP{{\mathcal P}}

\def\cS{{\mathcal S}}
\def\cT{{\mathcal T}}

\def\fA{\mathfrak{A}}
\def \fD{\mathfrak{D}}
\def\fB{\mathfrak{B}}

\usepackage{enumitem}

\renewcommand\labelenumi{(\roman{enumi})}
\renewcommand\theenumi\labelenumi

\def\Sb{\mathbf{S}}

\newcommand{\Var}[1]{\textrm{Var} \left[ #1 \right]}
\newcommand{\Cov}[1]{\textrm{Cov} \left( #1 \right)}
\newcommand{\pr}[1]{\mathbb{P} \left[ #1 \right]}
\newcommand{\er}[1]{\mathbb{E} \left[ #1 \right]}
\newcommand{\E}{\mathbb{E} }
\def \Gnp{G_{n,p}}
\newcommand{\tG}{\Gamma}

\def \Qone{Q} %this is denoted Q_1 in Christian's notes 

\def \nst{n_{*}}

\let\polishlcross=\l
\def\l{\ifmmode\ell\else\polishlcross\fi}

\makeatletter
\def\moverlay{\mathpalette\mov@rlay}
\def\mov@rlay#1#2{\leavevmode\vtop{   \baselineskip\z@skip \lineskiplimit-\maxdimen
	\ialign{\hfil$\m@th#1##$\hfil\cr#2\crcr}}}
\newcommand{\charfusion}[3][\mathord]{
#1{\ifx#1\mathop\vphantom{#2}\fi
	\mathpalette\mov@rlay{#2\cr#3}
}
\ifx#1\mathop\expandafter\displaylimits\fi}
\makeatother

\newcommand{\dcup}{\charfusion[\mathbin]{\cup}{\cdot}}

\newcommand{\vrhup}[1]{\scaleobj{0.6}{\scalerel*{\rightharpoonup}{#1}}}
\newcommand{\nrhup}{\mathord{\scaleobj{0.6}{\scalerel*{\rightharpoonup}{x}}}}
\newcommand{\wrhup}{\scaleobj{0.6}{\scalerel*{\rightharpoonup}{W}}}
\def\vseq#1{\ThisStyle{  \mathord{\vbox{\offinterlineskip\ialign{    \hfil##\hfil\cr
				$\SavedStyle{}_{\smash{\vrhup#1}}$\cr
				\noalign{\kern-0.7\scriptspace}
				$\SavedStyle#1$\cr}}}}}
\def\seq#1{\ThisStyle{  \mathord{\vbox{\offinterlineskip\ialign{    \hfil##\hfil\cr
				$\SavedStyle{}_{\smash{\nrhup}}$\cr
				\noalign{\kern-0.5\scriptspace}
				$\SavedStyle#1$\cr}}}}}
\def\wseq#1{\ThisStyle{  \mathord{\vbox{\offinterlineskip\ialign{    \hfil##\hfil\cr
				$\SavedStyle{}_{\smash{\wrhup#1}}$\cr
				\noalign{\kern-0.7\scriptspace}
				$\SavedStyle#1$\cr}}}}}

\def \absorber{\cA}

\let\setminus=\smallsetminus
\let\emptyset=\varnothing

\let\to=\lra

\makeatletter
\newcommand{\pushright}[1]{\ifmeasuring@#1\else\omit\hfill$\displaystyle#1$\fi\ignorespaces}
\newcommand{\pushleft}[1]{\ifmeasuring@#1\else\omit$\displaystyle#1$\hfill\fi\ignorespaces}
\makeatother

\let\N=\NN
\let\R=\RR

\let\Hc=\cH

\let\Pc=\cP
\let\Jn=J

\newcommand{\emb}{\textup{Emb}}

\begin{document}


\begin{abstract}
    A \emph{randomly perturbed graph} $G^p = G_\alpha \cup \Gnp$ is obtained by taking a~deterministic $n$-vertex graph $G_\alpha = (V, E)$ with minimum degree $\delta(G)\geq \alpha n$  
and adding the edges of the binomial random graph $\Gnp$ defined on the same vertex set $V$. For which value $p$ (depending on $\alpha$) does the graph $G^p$ contain a $K_r$-factor (a spanning collection of vertex-disjoint $K_r$-copies) with high probability?

    The  order of magnitude of the \textit{minimal} value of $p$ has been determined  whenever $\alpha \neq 1- \frac{s}{r}$ for an integer $s$ (see Han, Morris, and Treglown [RSA, 2021] and Balogh, Treglown, and Wagner [CPC, 2019]).
    
    We establish the minimal probability $p_s$ (up to a constant factor) for all values of $\alpha = 1-\frac{s}{r} \leq \frac 12$, and show that the \textit{threshold} exhibits a polynomial jump at $\alpha = 1-\frac{s}{r}$ compared to the surrounding intervals. An extremal example $G_{\alpha}$ which shows that $p_s$ is optimal up to a constant factor differs from the previous (usually multipartite) examples in containing a pseudorandom induced subgraph.
\end{abstract}

\maketitle

%************************* Introduction ****************************

\section{Introduction}

Much of research in extremal and probabilistic combinatorics revolves around establishing sufficient conditions for a graph $G$ to contain a specific (spanning) subgraph, such as a~Hamilton cycle, a specific spanning tree or  a $K_r$-factor (a spanning collection of vertex-disjoint $K_r$-copies), for instance~\cite{dirac52,Hajnal1970,jkv08,Posa1976}. The model of randomly perturbed graphs (introduced by Bohman, Frieze and Martin~\cite{bfm03,bfm04}) is designed to interpolate between the worst-case scenario for the host graph $G$, and a random choice of $G$. Their model is also inspired by \textit{smoothed analysis} of algorithms, which follows the same principle~\cite{st04}. Numerous questions and interesting transition phenomena have been studied in this model, for instance~\cite{adrrs21,balog16,bhkm19,bmpp20,jk20,kks17,kst06}. One reason for its appeal is that the proofs often elegantly combine tools from extremal graph theory with probabilistic arguments. Following~\cite{balog16,bpss23,hmt21}, we will be investigating sufficient conditions for finding a $K_r$-factor in the randomly perturbed model.\footnote{Randomly \textit{augmented} would be the technically more accurate term~\cite{adrrs21,adr23}, since all the cited articles are considering  one-sided perturbations: edges are only added to $G_\alpha$, and never removed. However, the term \textit{perturbed} has prevailed in the literature and we use it henceforth.}

A \emph{randomly perturbed graph} $G^p = G_\alpha \cup \Gnp$ is obtained by taking a deterministic $n$-vertex graph $G_\alpha = (V, E)$ with minimum degree $\delta(G)\geq \alpha n$  %(typically satisfying a deterministic graph property such as a lower bound on the minimum degree of a vertex) 
and adding the edges of the binomial random graph $\Gnp$ defined on the same vertex set $V$. The threshold probability for a $K_r$-factor is the lowest value of $p$ which with high probability ensures a $K_r$-factor in $G^p$, for any deterministic graph $G_{\alpha}$. Here, \textit{with high probability} (w.h.p.)\ means with probability tending to 1 with $n\rightarrow\infty$. %However, many other properties of this model have been investigated (see~\cite{bfm12, drrs19, bpss21}). 
Throughout the introduction, $n$ will be the number of vertices of our graph, and we will assume that it is divisible by $r$.

As mentioned, the model of randomly perturbed graphs bridges the gap between Dirac-type problems and the corresponding thresholds in $\Gnp$ in a natural way. The two ``extreme cases'' ($\alpha = \left(1-\frac{1}{r} \right)$ and $\alpha = 0$) cover classical and difficult results in extremal and probabilistic combinatorics. We remark that the case $r=2$ corresponds to perfect matchings, and the case $r=3$ to triangle factors, which have been studied much earlier (see \cite{ch63,dirac52}). Proving a~conjecture of Erd\H os, Hajnal and Szemer\'edi \cite{Hajnal1970} showed that  for $\alpha \geq \left(1-\frac{1}{r} \right)$ and sufficiently large $n$, any graph $G_\alpha$ with minimum degree $\alpha n$ contains a $K_r$-factor, so no random edges are needed (or, formally, $p=0$ suffices). An example showing that a smaller value of $\alpha$ does not suffice is a graph consisting of an independent set of size $n/r + 1$ and all the remaining edges. On the other hand, a celebrated result of Johansson, Kahn and Vu~\cite{jkv08} states that for $p = Cn^{-2/r}(\log n)^{2/(r(r-1))}$, $\Gnp$ contains a $K_r$-factor with high probability. The fact that this is the correct order of magnitude for the threshold probability follows from the necessary condition that every vertex is contained in a copy of $K_r$. Another phenomenon which is ubiquitous in this area is that if one only requires $(1-\eps)n/r$ vertex-disjoint $K_r$ copies in $\Gnp$, the probability $Cn^{-2/r}$ already suffices~\cite{rucinski92}.

The threshold for a $K_r$-factor in the randomly perturbed model $G_\alpha \cup \Gnp$ has been determined for all but finitely many values of $\alpha$ by Han, Morris and Treglown~\cite{hmt21}. The  case of $\alpha \in (0,1/r)$ had been resolved earlier by Balogh, Treglown and Wagner~\cite{btw18}.

\begin{thm}
    [Clique-factors in randomly perturbed graphs~\cite{btw18}, \cite{hmt21}]
    \label{thm:hmt}
    For any integers $1 < s \leq r$ and any $\alpha \in \left(1-\frac{s}{r}, 1-\frac{s-1}{r} \right)$, there exists $C$ such that the following holds. For any $n$-vertex graph $G$ with minimum degree at least $\alpha n$, w.h.p.\ $G \cup \Gnp$ with $p = Cn^{-2/s}$ contains a $K_r$-factor.
\end{thm}

Let us describe the extremal construction (depicted in Figure~\ref{fig:extremal_1} for $r=4$ and $s=2$) which implies that Theorem~\ref{thm:hmt} is optimal up to a constant factor. Let $\alpha = 1-\frac sr + \gamma$ for $\gamma \in \left(0, \frac 1r\right)$, and let $F_{r,s,\gamma}$ be the graph consisting of an independent set $A$ of order $|A| = \left(\frac{s}{r}-\gamma \right) n$, and all the remaining edges. We have $\delta(F_{r,s,\gamma}) \geq \left(1-\frac sr + \gamma\right)n$. Any $K_r$-factor in $F_{r,s,\gamma}$ would have to use $\Omega(n)$ copies of $K_s$ in $\Gnp[A]$ (as otherwise the $K_r$-factor would use at most $\frac{(s-1)n +o(n)}{r} < |A|$ vertices from $A$), which typically cannot be found  for $p = cn^{-2/s}$ with a small constant $c$. Hence, w.h.p.\,$F_{r,s,\gamma}\cup \Gnp$ does not contain a $K_r$-factor.

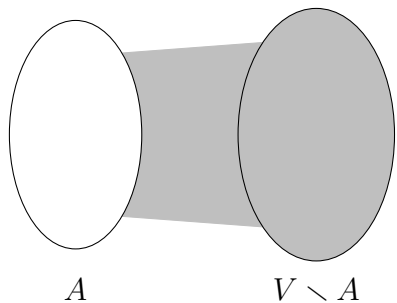
\begin{figure}
    \begin{tikzpicture}[scale=0.8]
        \centering
        \filldraw[color=lightgray] (0,1.3) -- (4,1.6) -- (4,-1.6) -- (0,-1.3) -- (0,1.3);
        \draw[fill=white] (0,0) ellipse (1.1cm and 1.9cm);
        \draw[fill=lightgray] (4,0) ellipse (1.3cm and 2.1cm);
        \node at (0,-2.2) [below] {$A$};
        \node at (4,-2.2) [below] {$V\setminus A$};
    \end{tikzpicture}
\caption{\label{fig:extremal_1}The extremal construction $F_{4,2,\gamma}$. The set $A$ with $|A|=\left(\frac12-\gamma\right)n$ is an independent set in $F_{4,2,\gamma}$. Any $K_4$-factor in $F_{4,2,\gamma}$ would have to use $\Omega(n)$ edges  in $\Gnp[A]$.}
\end{figure}

In case $\alpha = 1- \frac sr$ (i.e.~when $\gamma = 0$), one might expect that an extra logarithmic factor in the threshold probability is needed, analogously to the above-mentioned result of Johansson, Kahn and Vu. Indeed, informally speaking, for $p = cn^{-2/s}(\log n)^{2/(s(s-1))}$ one can show that any $K_r$-factor in $F_{r,s,0}\cup \Gnp$  would \textit{essentially}\footnote{that is, with $polylog(n)$ exceptional vertices} require a perfect $K_s$-factor in $\Gnp[A]$, which again typically does not occur for small $c$. As mentioned before, such phenomena are ubiquitous in the study of random graphs, and have also been observed in other problems on randomly perturbed graphs (squares of Hamilton cycles~\cite{bpss24}, and triangle factors~\cite{bpss23}).

However, B\"ottcher et al.~\cite{bpss23} have pointed out a rather surprising \textit{polynomial} jump in the threshold probability when $\gamma = 0$, $r=4$ and $s=3$ (so $\alpha = \frac 14$), and asked about the correct threshold. We show that this threshold is of order $n^{-3/5}$ --  strictly between $n^{-2/4}$ and $n^{-2/3}$ -- which is the order of magnitude of the threshold probability for $\alpha > \frac 14$ and $\alpha < \frac 14$, respectively. The intuition for such a threshold comes from the following example. Let $ G=G_{1/4}$ be some graph which consists of a set $A$ of size $\frac {3n}{4} +g$ such that $G[A]$ has minimum degree $g$, and all the edges not having both endpoints in $A$. Any $K_4$-factor in $G_{1/4}$ would require at least $g$ copies of $K_4$ in $G[A]$, and this can take as many as $\Omega(n^{-3/5+2})$ random edges for an appropriate choice of $g$ and $G$ (see Theorem~\ref{t:lower-bound} and its proof). %that this is the probability at which, for any $n$-vertex graph $G$ with $\delta(G) \geq g$ (specify $g$), the graph $G \cup \Gnp$ \textit{typically} contains $g$ vertex-disjoint copies of $K_{s+1}$.
  The construction from the concluding remarks of~\cite{bpss23} already exhibits this \textit{polynomial jump} in threshold probability, but they did not find a construction for $G_{1/4}$ yielding the correct threshold of order $n^{-3/5}$.

Our main result is establishing the above-mentioned threshold, and a generalisation of the phenomenon for any $r$ and $s \geq r/2$. Before we state it, let us introduce the following functions. Let 
\[\varphi(s) = \frac{2s}{(s-1)(s+2)}\]
and 
\[p_s = p_s(n) = \begin{cases} 
	\frac{\log n}{n} & \text{ if } s=2 \\
	n^{-\varphi(s)} & \text{ if } s\geq 3.
\end{cases} \]

Note that $p_s > n^{-2/s+\eps} $ for $\eps$ sufficiently small.

\begin{thm}\label{t:main}
	For any integers $r$ and $s$ with $r/2 \leq s < r$, there exists $C>0$ such that the following holds. Let $G$ be an $n$-vertex graph with minimum degree at least $\left(1 - \frac{s}{r} \right)n$. For $n$ divisible by $r$ and $p = Cp_s(n)$, w.h.p.\,the graph $G^p = G \cup \Gnp$ has a~$K_r$-factor.
\end{thm}
This answers questions raised in~\cite{bpss23,hmt21}. 

As for the threshold lower bound, the property that ends up being crucial and heavily used throughout the argument is that for $\alpha = o(1)$ and $p = Cp_s$, the graph $G_\alpha \cup \Gnp$ typically contains $\alpha n$ disjoint copies of the complete graph $K_{s+1}$ (see Lemma~\ref{l:5.5chr}).
One very interesting feature of this problem is that to show the lower bound for the lemma above, the construction for the \textit{host graph} $G_\alpha$ uses partly a pseudorandom graph. This is completely different to all the extremal constructions we have encountered when dealing with randomly perturbed graphs, in which $G_{\alpha}$ is usually some highly structured, e.g.~complete multipartite, graph.
In fact, we prove a lower bound for any choice of parameters $s$ and $r$, and conjecture that this is the correct threshold.

\begin{thm}\label{t:lower-bound}
    %Let $r = ms+t$ with $0 < t \leq s$, 
    Let $1 < s <r$, and let $n$ be divisible by $r$. For any $\eps>0$, there is a constant $c$ and a graph $G$ with minimum degree  $\left(1 - \frac{s}{r} \right)n$ such that for $p = cp_s(n)$, the probability that the graph $G^p = G \cup \Gnp$ has a $K_r$-factor is less than $\eps$.
\end{thm}

Let us finish the introduction by pointing out the synergy between extremal combinatorics highlighted by our proof  [before: that the proofs in this area are interesting as they use a combination of extremal combinatorics and random graph theory]. In particular, in the deterministic part of our graph, we need to find a partial $K_{s,r-s}$-factor with $o(n) $ leftover vertices, even when the minimum degree is $(1-s/r-\delta) n$. This is implied by results from Section~\ref{sec:6-chr}, in which we also prove a stability version of some Dirac-type results due to Shokoufandeh and Zhao~\cite{sz03}. These results and methods may be of independent interest.

Another technique worth pointing out is an application of the (now proven) fractional Kahn--Kalai conjecture which turns out to be naturally compatible with the above-mentioned results from extremal graph theory and the regularity method.

Finally, we believe that the conclusion of Theorem~\ref{t:main} also holds for all $s  \in \{2, 3, \ldots, \lfloor (r-1)/2 \rfloor \}$, and we are hoping to prove the full statement soon. %\nk{[Let's talk about how to formulate this. I initially had this in the Concluding Remarks (now commented), but it really depends on whether we think it's ok to claim this question for ourselves.]}

The paper is structured as follows. Section~\ref{sec:lower-bound} contains the proof of Theorem~\ref{t:lower-bound} (the construction of the extremal graph and the lower bound for the threshold probability). Section~\ref{sec:main-proof} contains the statements of the main lemmas, the deduction of Theorem~\ref{t:main}, and the distinction of two main cases (\textit{extremal} vs.\,\textit{non-extremal}) depending on the structure of the deterministic graph $G_\alpha$. The main steps of the proof are also outlined in Section~\ref{sec:main-proof}. The extremal case and the bottleneck-problem of finding \textit{many} $K_{s+1}$-copies are treated in Section~\ref{sec:larger-clques}. Sections~\ref{sec:spread-applications}--\ref{sec:non-extremal} contain the proof in the non-extremal case, which includes the above-mentioned tools. %applications of the fractional Kahn--Kalai conjecture and auxiliary extremal results. 
Since the non-extremal case is rather involved, we have also included a short proof outline in Section~\ref{sec:non-extremal}.

%************************* Preliminaries ****************************

\section{Preliminaries}

Let $t$ be such that $r = s + t$. Throughout the paper, we refer to the case $s = t = r/2$ as the \textit{singular} case, and in this case some steps are treated separately. 

For a graph $G=(V,E)$ we denote by $V(G)$ and $E(G)$ the set of vertices and the set of edges of $G$, and by $v(G)$ and $e(G)$ the sizes of the respective sets. For a subset $S\subset V$, by $e_G(S)$ we denote the number of edges spanned in $G$ by vertices of $S$, and by $N_G(S)$ we denote the set of common neighbours in graph $G$ of vertices in $S$. Next, $G(V,p)$ denotes the usual binomial random graph defined specifically on the vertex set $V$, and 
$G_{n,m,p}$ is the random bipartite graph with partition classes of size $n$ and $m$, respectively. Moreover, $G[V_1,V_2]$ is the bipartite subgraph of $G$ induced by partition classes $V_1$ and $V_2$.

%An \textit{embedding} of a graph $F$ is an injective graph homomorphism from $V(F)$ to $V(G)$. 
For a family of graphs $\mathcal{P}$, a \emph{$\mathcal{P}$-factor} in a graph $G$ is a collection $\mathcal{F}$ of pairwise vertex-disjoint subgraphs of $G$, each isomorphic to an element of $\mathcal{P}$, which together span $V(G)$. Sometimes we refer to a collection of such pairwise vertex-disjoint subgraphs of $G$ as a \textit{partial factor}.

For any set $X$, by ${X \choose \ell}$ we denote the family of all $\ell$-element subsets of $X$.

\subsection{Probabilistic bounds}

\begin{lemma}\label{l:ks-in-lin-sized}
    For $s \geq 3$ and $\eps >0$, let  $p = n^{-2/s+\eps}$. Then, w.h.p.\,any subset $S \subset V(\Gnp)$ with $|S| \geq \frac{n}{\log^2 n}$ contains a copy of $K_s$.
\end{lemma}

\begin{proof}
    Let $n_* = n / \log^2 n$. It is enough to consider subsets $S \subset V(\Gnp)$ with $|S|=n_*$. For a fixed vertex set $S$ note that the subgraph of $\Gnp$ induced by $S$ is distributed as a~random graph with edge probability $p$. Thus, if $\mu_{n_*}$ is the expected number of copies of $K_s$ in $\Gnp[S]$, we have
    $$\frac{\mu_{n_*}}{n} \geq \frac{n_*^s}{s^s } p^{\binom s2} =        \frac{n_*}{s^s } (n_* p^{\frac{s}2})^{s-1}
    \geq \frac{n}{\log ^2 n} \cdot n^{\eps s(s-1)/4} >n$$
    	%$$\frac{\mu_{n_*}}{n} \geq \frac{1}{s^s \log^2 n} n_*^{s-1} p^{\binom s2} =        \frac{1}{s^s \log^2 n} (n_* p^{\frac{s}2})^{s-1} \geq n^{\eps s(s-1)/4} >1$$
    for sufficiently large $n$.
    Hence, by Janson's inequality (e.g.,~\cite[Theorem 2.14]{jlr00}, originally proved in~\cite{janson90}),
    $$\pr{\Gnp[S] \not \supset K_s} = e^{- \Omega(n)}. $$
    Taking the union bound over all possible choices of $S\subset V(\Gnp)$ with $|S|=n_*$, we have 
    $$\sum_S \pr{K_s \not \subset \Gnp[S]} \leq \left( e \log^2 n \right)^{n_*} e^{- \Omega(n)} = o(1), $$
    which completes the proof.
\end{proof}

\subsection{Using the Fractional Kahn--Kalai conjecture}
    \label{sec:prelim-kk}
    We will apply the powerful result of Frankston, Kahn, Narayanan and Park~\cite{fknp21}, which confirmed the fractional Kahn--Kalai conjecture. The full conjecture has been confirmed in 2022 by Park and Pham~\cite{pp22}. For an introduction to this ground-breaking tool, we recommend the recent survey of Perkins~\cite{perkins2024searching}.
 See also~\cite{pp22} for the proof of the full conjecture. 
 
    Consider a finite ground set $Z$ and fix a nonempty collection of subsets $\cH \subset 2^Z$. For example, $Z$ can consist of the edges of the complete graph $K_n$ and $\cH$ can be the collection of perfect matchings in $K_n$. The fractional Kahn--Kalai conjecture  reduces the probabilistic problem of finding an element of $\cH$ in a random subset of $Z$ to an `extremal' problem of finding a so-called spread measure on $\cH$.

For $q >0$, a probability measure $\mu$ on $\cH$ is called a \emph{$q$-spread} if for every set $S \subset Z$,
\begin{equation}
	\label{eq:def-spread}
	\mu(\{A \in \cH: A \supseteq S\}) \leq q^{|S|}.
\end{equation}
We write $Z_p$ for a random subset of $Z$ in which each element is included independently with probability $p \in [0,1]$. It is not difficult to show (see e.g.~Frankston et al.~\cite{fknp21}) that the existence of a $q$-spread measure is necessary for $Z_q$ to contain an element of $\Hc$ with probability, say, $\frac 12$. The following theorem shows that up to a logarithmic factor, it is also sufficient for finding w.h.p.\,an element of $\Hc$. We state the theorem in a packaged form from~\cite{psss22}.

\begin{thm}[\cite{fknp21}, \cite{psss22}]\label{thm:fractional-kk} 
There exists a constant $C>0$ such that the following holds. Consider a non-empty ground set $Z$ and fix a non-empty collection of subsets $\cH \subset 2^ Z$. Suppose that there exists a $q$-spread probability measure on $\cH$, and let $p=\min(Cq \log |Z|, 1)$. Then the probability that $Z_p$ does not contain an element of $\cH$ tends to $0$ as $|Z| \rightarrow \infty$.
\end{thm}

Returning to the above-mentioned example, if $\cH$ is the collection of perfect matchings in $K_n$, then the uniform measure on $\cH$ is an $O(n^{-1})$-spread measure on $\cH$. So Theorem~\ref{thm:fractional-kk} implies the standard (but non-trivial) result  that the random graph $\Gnp$ with $p \geq C n^{-1} \log n$ w.h.p.\,contains a perfect matching for some $C>0$. A more interesting example of a spread measure is Lemma~\ref{l:two-spread}.

Our main application of Theorem~\ref{thm:fractional-kk} will be to find $K_r$-factors in randomly perturbed regular pairs (see Lemma~\ref{c:st-reg-pairs-factors}), and to this end we utilise the ideas from~\cite{psss22}. We also give a simpler application, Lemma~\ref{l:x-cover}, which is
needed for covering the ‘leftover vertices’ after regularising the graph.

In case $\cH$ is a collection of $K_s$-factors for $s >2$, the spread condition~\eqref{eq:def-spread} is slightly more difficult to verify, so it is convenient to pass to the setting of $s$-uniform hypergraphs, which is also done in~\cite{psss22}. In our applications, the ground set $Z$ will be an $s$-uniform hypergraph ($s$-graph), so given a graph $F$, let $F^{(s)}$ denote the collection of $K_s$-copies in $F$. For an $s$-graph $H$, let $\partial H$ denote the set of pairs contained in elements of $H$ (i.e.,~the 1-skeleton of $H$). We will be constructing spread measures on $F^{(s)}$ (for suitable $F$), which significantly simplifies the spread calculation~\eqref{eq:def-spread}, and then apply beautiful results due to Heckel~\cite{heckel21} and Riordan~\cite{riordan22} which immediately imply the desired threshold in random graphs.\footnote{Passing to hypergraphs also saves a $polylog$-factor in the threshold probability, but this is not needed for our application.}
They constructed a coupling between the random graph $G_{n,p}$ and a~random $s$-uniform hypergraph $H^{(s)}(n,\pi)$ with edge probability $\pi \sim p^{\binom s2}$, in a way that the edges of $H^{(s)}(n,\pi)$ form cliques in $G_{n,p}$. We state these results in the following lemma, the proof of which can be found in~\cite{psss22}, where the ground sets $K_n$ and $K_n^{(s)}$ are replaced by $G$ and $G^{(s)}$, respectively.

In the following lemma, $G_p$ denotes a random subgraph of $G$ in which each edge is present independently with probability $p$, and $G^{(s)}_\pi$ denotes a similar random subhypergraph of $G^{(s)}$ with edge probability $\pi$. 

\begin{lemma}[\cite{heckel21,riordan22}]\label{l:coupling}
Let $s \geq 3$, and let $\eps=\eps(s)$ be sufficiently small. For any $p \leq n^{-2/s+\eps}$ and for some $\pi \sim p^{\binom s2}$, there is a joint distribution $\lambda$ of a graph $F$ and an $s$-uniform hypergraph $H$ such that w.h.p.\,$H$ is contained in $F^{(s)}$ (that is, every hyperedge of $H$ corresponds to an $s$-clique in $F$). Moreover, the marginal distribution of $F$ is the same as that of $G_p$, and the marginal distribution of $H$ is the same as that of $G^{(s)}_{\pi}$. 
\end{lemma}

\subsection{Regularity}\label{sec:prelim-reg}
Some of our proofs rely heavily on the regularity method. One phenomenon that our paper, as well as~\cite{psss22}, demonstrates, is that the concept of regular pairs works very well in conjunction with the machinery of spread measures introduced in Section~\ref{sec:prelim-kk}.

First we introduce some necessary terminology. Let $G$ be a graph and let $V_1, V_2 \subset V(G)$ be disjoint subsets of the vertices of $G$. For non-empty sets $X_1\subseteq V_1$, $X_2 \subseteq V_2$, we define the \emph{density of $G[X_1,X_2]$} as $d_G(X_1, X_2)\coloneq \tfrac{e_G(X_1,X_2)}{|X_1||X_2|}$.
Given $\eps, d > 0$, we say that a pair $(V_1,V_2)$ is \emph{$(\eps,d)$-regular} in $G$ if for $i = 1,2$ and for all sets $X_i \subseteq V_i$  with $|X_i|\geq \eps |V_i|$, we have $| d_G(X_1,X_2)-d| < \eps$.
We usually omit the graph $G$ if it is clear from the context, for instance in the statement of Lemma~\ref{l:slicing} below.
    
We will use the following standard lemmas, which can be found for instance in~\cite{abcd22,ks95}.

\begin{lemma}\label{l:slicing}
    Let $0 < \eps < \beta$, $d \leq 1$ and let $(V_1,V_2)$ be an $(\eps,d)$-regular pair. Then for any pair $(U_1, U_2)$ such that for $i=1,2$, $U_i\subset V_i$ with $|U_i| \geq \beta |V_i|$, the pair $(U_1,U_2)$ is $(\eps', d')$-regular with $\eps' = \max\{\eps/\beta, 2\eps \}$ and some $d'>0$ satisfying $|d' -d|\leq \eps$.
\end{lemma}

To state the version of the Regularity Lemma which we will apply (see~\cite[Lemma 2.6]{abcd22}), we introduce the notion of regular partitions and the reduced graph.
We say that the partition $V_0\dcup V_1 \dcup \ldots \dcup V_\ell$ of $V(G)$ is an \emph{$\eps$-regular partition} if $|V_0| \leq \eps |V(G)|$, $|V_1| = \ldots = |V_{\ell}|$, and for all but at most $\eps \ell^2$ pairs $(i,j) \in [\ell] \times [\ell]$, $i\neq j$, the pair $(V_i, V_j)$ is $\eps$-regular. The \emph{$(\eps, d)$-reduced graph} $R$ with respect to the above partition is a graph on vertex set $[\ell]$ such that the edges $ij$ correspond precisely to the $(\eps, d')$-regular pairs $(V_i,V_j)$ with density $d' \geq d$.

\begin{lemma}\label{l:regularity}
    For all $0 < \eps \leq 1$ and $\ell_0 \in \N$, there exists $L \in \N$ such that for every $0<d<\gamma<1$, every graph $G$ on $n >M_0$ vertices with minimum degree $\delta(G)\geq \gamma n$ has an $\eps$-regular partition $V_0 \dcup V_1 \dcup \ldots \dcup V_\ell$ with $(\eps,d)$-reduced graph $R$ on $\ell$ vertices such that $\ell_0 \leq \ell \leq L$ and $\delta(R)\geq (\gamma - d - 2\eps)\ell$.
\end{lemma}

Next, to find spanning subgraphs, it is convenient to strengthen the notion of regular pairs to super-regular pairs.   
A pair of sets $(V_1, V_2)$ in a graph $G$ is \emph{$(\eps, d, \vartheta)$-super-regular} if it is $(\eps, d)$-regular, and for $i \neq j$, every vertex $v \in V_i$ has at least $\vartheta|V_j|$ neighbours in $V_j$. As usual, we refer to an $(\eps, d, d-\eps)$-super-regular pair as \emph{$(\eps, d)$-super-regular}.

The following lemma allows us to pass from $(\eps, d, \vartheta)$-super-regular to $(\eps', \vartheta, \vartheta-\eps')$-super-regular pairs by discarding some edges (see~\cite[Lemma 2.12]{abcd22}).

\begin{lemma}[\cite{abcd22}]\label{l:super-reg-min}
    Let $0<\eps<1$ and let $G$ be a bipartite graph with parts $V_1, V_2$ of size $n$, for sufficiently large $n$. Let $\vartheta \in [4 \eps, 1]$. If $(V_1, V_2)$ is $(\eps^2,d', \vartheta-\eps^2)$-super-regular, then there is a spanning subgraph $G' \subset G$ such that $(V_1,V_2)$ is~$(4\eps,\vartheta, \vartheta - 4\eps)$-super-regular in~$G'$.
\end{lemma}

Finally, let us state the Blow-up Lemma for the special case that we will apply.

\begin{lemma}[\cite{kss97}]\label{l:blow-up} 
    Given $r, \vartheta >0$, there exists $\eps >0$ such that the following holds. Let $V_1,\ldots, V_r$ be vertex sets of a graph $G$ with $|V_i| =n$, $i=1,2,\ldots,r$, where $n$ is sufficiently large. If each pair $(V_i, V_j)$ (where $1 \leq i < j \leq r$) is $(\eps, d, \vartheta)$-super-regular for some $d \geq \vartheta$, then $G$ contains a $K_r$-factor in which each copy of $K_r$ uses one vertex of each part $V_i$.
\end{lemma}

%********************* Lower bound for the threshold probability ************************

\section{Lower bound for the threshold probability}
\label{sec:lower-bound}

We will now prove Theorem~\ref{t:lower-bound} which gives lower bound for the threshold probability. The lower bound will be proved in more generality, for all $r = ms+t$ with $0 < t \leq s$ and $m$ an integer. Our task is to define a family of $n$-vertex graphs $G$ satisfying  $\delta(G) \geq \left(1-\frac{s}{r}\right)n$ and such that for some constant $c$ and $p \leq c p_s$ the probability that the graph $G^p = G \cup \Gnp$ contains a $K_r$-factor is \textit{small}.

\subsection{The case $s=2$}

Fix $c<1$ and let $p=p(n)=c\frac{\log n}{n}$. For $n$, a large multiple of $r$, let $G$ be the complete $n$-vertex $\lceil\frac{r}2\rceil$-partite graph with $\lfloor\frac{r}2\rfloor$ vertex classes of size $\frac{2n}{r}$ and, if $r$ is odd, one further class of size $\frac{n}{r}$. Clearly $\delta(G)=n-\frac{2n}{r}=\left(1-\frac{2}{r}\right)n$. Assume for contradiction that $G^p=G\cup G_{n,p}$ has a~$K_r$-factor $\mathcal{F}$. Let $A$ be one of the vertex classes of size $\frac{2n}{r}$ and for each clique $K\in\mathcal{F}$ set
\[W(K) = |V(K) \cap A|-2.\]
Note that 
\[ \sum_{K\in\mathcal{F}} W(K) = |A| - 2|\mathcal{F}| = 0.\]
Denote by $\mathcal{F}^+$ the family of those cliques $K\in\mathcal{F}$ for which $W(K) > 0$, and by $\mathcal{F}^-$ the family of those for which $W(K) < 0$, and observe that
\[\sum_{K\in\mathcal{F}^+} |W(K)| = \sum_{K\in\mathcal{F}^-} |W(K)|. \]
The family $\mathcal{F}^+$ consists of exactly those cliques $K$ which have at least three vertices in $A$, hence they induce at least one copy of $K_3$ in $G^+[A]$, and by definition of $G$ all edges of such $K_3$ need to come from the random graph $G_{n,p}$.

Let $X$ denote the number of isolated vertices in $G_{n,p}[A]$. Then
\[ \E X = \frac{2n}{r}\left(1-c\frac{\log n}{n}\right)^{\frac{2n}r-1}=\Theta\left(n^{1-\frac{2c}{r}}\right).\]
Hence, by Markov's inequality, we get that $A$ has at least $\Theta(n^{1-\frac{c}{r}})$ isolated vertices in $G_{n,p}$. Each such vertex must be covered by some clique $K\in\mathcal{F}^{-}$. Thus 
\[ \sum_{K\in\mathcal{F}^{-}} W(K) \geq \Theta(n^{1-\frac{c}{r}}). \]

On the other hand, using again Markov's inequality one can show that w.h.p.\,the number of copies of $K_3$ in $G_{n,p}[A]$ is at most $O(\log^3 n)$ and the number of copies of $K_4$ in $G_{n,p}[A]$ is $o(1)$, whence
\[ \sum_{K\in\mathcal{F}^{+}} W(K) = O(\log^3 n),\]
and we reach a contradiction.

\subsection{The case $3 \leq s < r$}

For two graphs $F$ and $H$ let $\emb(F,H)$ denote the family of all embeddings from $F$ to $H$, that is injective maps $\psi:V(F)\to V(H)$ which preserve adjacency (i.e.\ for any $x,y\in V(F)$, $xy\in E(F)$ if and only if $\psi(x)\psi(y)\in E(H)$). 

\begin{lemma}\label{lem:emb}
    For every integer $\ell > 0$ there are $L, n_0 > 0$ such that for any $n \geq n_0$ and $k$ with $(\log n)^2 \leq k < \frac{n}{3}$ there exists a bipartite $n$-vertex graph $H$ with $\delta(H)\geq k$ and such that 
    \[ |\emb(F,H)| \leq Ln^{v(F)}\left(\frac{k}{n}\right)^{e(F)}\]
    for any graph $F$ with $v(F)\leq \ell$. 
\end{lemma}

\begin{proof}
    Let $H$ be a random bipartite graph $G_{\frac{n}2,\frac{n}2,\frac{3k}{n}}$. Then for any vertex $v$ in $H$ the expected number of edges incident with $v$ is $\frac32k$, and by Chernoff bound the probability that there is at most $k$ edges incident with $v$ is bounded from above by $exp\left(-\frac{k}{12}\right)$. Taking a union bound, we get that with probability $1-o(1)$ we have $\delta(H)\geq k$.
	
    Now let $F$ be a graph with $v(F)\leq \ell$. Then
    \[ \E |\emb(F,H)| = (n)_{v(F)} \left(\frac{3k}{n}\right)^{e(F)}\left(1-\frac{3k}{n}\right)^{{v(F)\choose 2}-e(F)} = O\left(n^{v(F)}\left(\frac{k}{n}\right)^{e(F)}\right).\]
	
    Taking $L$ sufficiently large and using Markov's inequality, we can ensure that with positive probability for each $F$ with $v(F)\leq \ell$ the number of embeddings in $H$ is at most $Ln^{v(F)}\left(\frac{k}{n}\right)^{e(F)}$.
\end{proof}

We will now construct the graph $G$. For a fixed $s$ with $3\leq s < r$, let $n$ be a large multiple of $r$ and $k=n^{1-\varphi(s)}$. Let $V$ be a set of size $n$, and $A \subset V$ with $|A| = \frac{sn}{r}+k$. Let $G$ be a graph on vertex set $V$ obtained in the following way. For $H$, a bipartite graph given by Lemma~\ref{lem:emb} with $\ell, k, n$ equal to, respectively, $s+1$, $n^{1-\varphi(s)}$ and $\frac{sn}{r}+k$, we take $G$ so that $G[A]$ is isomorphic to $H$, $G[V\setminus A]$ is the complete graph and $G[A,V\setminus A]$ is the complete bipartite graph.
Then $G$ satisfies the following conditions:
\begin{itemize}
	\item[(G1)] ${V\choose 2} \setminus {A\choose 2} \subseteq E(G)$, 
	\item[(G2)] $\delta(G[A]) \geq k$,
	\item[(G3)] $G[A]$ contains no copy of $K_{s+1}$, since $s\geq 3$ and $H$ is bipartite,
	\item[(G4)] if $F$ is a graph with $v(F)=s+1$, then $|\emb(F,G[A])| \leq Ln^{s+1}\left(\frac{k}{n}\right)^{e(F)}$. 
\end{itemize}

Clearly, by (G1) and (G2) we have $\delta(G)\geq k + |V\setminus A| = n -\frac{s}{r}n = \left(1-\frac{s}{r}\right)n$. Take $L \gg r$ at least as large as $L$ given by Lemma~\ref{lem:emb} with $\ell=s+1$, $c \ll L^{-1}$,
 and  $p=cn^{-\varphi(s)}=\frac{ck}{n}$. Assume for contradiction that $G^p=G\cup G_{n,p}$ has a $K_r$-factor $\cF$. First notice that $G^p[A]$ must contain at least $\Omega(k)$ copies of $K_{s+1}$. Indeed, each copy of $K_r$ in $\cF$ which doesn't yield a copy of $K_{s+1}$ in $G^p[A]$ can cover at most $s$ vertices in $A$ and thus have to cover at least $r-s$ vertices in $V\setminus A$. Therefore, the number of vertices in $A$ covered by such copies of $K_r$ is at most
$ \frac{s|V\setminus A|}{r-s} = \frac{sn}{r} - \frac{sk}{r-s}.$ The remaining at least $\frac{kr}{r-s}$ vertices must belong to at least one copy of $K_{s+1}$ in $G^p[A]$, which give us at least $\frac{kr}{(r-s)(s+1)}$ such copies. Next, each of them arises as the union of some graph $F\neq K_{s+1}$, $v(F)=s+1$, and ${s+1\choose 2}-e(F)>0$ random edges from $G_{n,p}$. However, the number of copies of $K_{s+1}$ in $G^p[A]$ arising in this way has expectation bounded from above by
\begin{align*}
	\sum_{F}|\emb(F,G[A])|\, p^{{s+1\choose 2}-e(F)} & \leq \sum_F Ln^{s+1}\left(\frac{k}{n}\right)^{e(F)}\left(\frac{ck}{n}\right)^{{s+1\choose 2}-e(F)}  \\
	& \leq \sum_F c Ln^{s+1}\left(\frac{k}{n}\right)^{{s+1\choose 2}} \leq 2^{{s+1\choose 2}}cLn^{s+1-{s+1\choose 2}\varphi(s)}.
\end{align*}
Note that
\[ s+1 - {s+1\choose 2}\varphi(s) = (s+1)\left(1 - \frac{s^2}{(s-1)(s+2)} \right) = \frac{s^2-s-2}{(s-1)(s+2) = 1-\varphi(s)},
\]
hence the above upper bound is equal to
\[ 2^{{s+1\choose 2}}cLn^{1-\varphi(s)} = 2^{{s+1\choose 2}}cLk. \]
Finally, by Markov's inequality, with probability at least $1-\frac{1}{L}$, there are at most $2^{{s+1\choose 2}}cL^2k$ copies of $K_{s+1}$ in $G^p[A]$. Taking $c$ sufficiently small with respect to $L$, we can make this number smaller than $\frac{kr}{(r-s)(s+1)}$, getting a contradiction. Thus, taking $L = 1/\eps$ (and $c\ll L^{-1}$) the probability that $G^p$ has a $K_r$-factor is at most $\eps$, which concludes the proof of Theorem~\ref{t:lower-bound}.

%************************* proof of main result ****************************

%\section{The proof for $2s \geq r$ modulo main lemmas}
\section{The proof of the main result -- upper bound for the threshold probability}
\label{sec:main-proof}
%Let $\alpha = 1 - \frac{s}{r}.$ Assuming that $2s \geq r$, we have $s = r-k \geq k$
Recall that we wish to construct a $K_r$-factor in $G^p = G \cup \Gnp$, where $G$ is an $n$-vertex graph with $n$ a large multiple of $r$,
\[\delta(G) \geq \left( 1 - \frac{s}{r}\right) n \ \ \text{ and } \ \ p = Cp_s(n).\] 

In order to prove the main result we will distinguish two cases for the graph $G$, given by the following lemma the proof of which can be found in Section~\ref{sec:cases}. %This lemma will be applied with $\alpha = 1 - \frac sr$. 

\begin{lemma}\label{l:4.3chr-simplified}
    Let $\alpha \in (0, 1)$, and let $\gamma, \beta$ be positive constants satisfying $8\gamma < \beta < \frac 15(1-\alpha)$.  If $\delta(G) \geq \alpha n$, then one of the following holds in $G$.
    \begin{enumerate}
	\item[(a)] There is a partition $A_1 \dcup A_2 = V(G)$ with $|A_1| = (1-\alpha \pm \gamma ) n$ such that 
	\begin{itemize}
		\item[(i)] for $x \in A_1$, $|A_2 \setminus N_G(x)| \leq 4\beta n$, 
		\item[(ii)] for $x \in A_2$, $|N_G(x)\cap A_1| \geq \beta n$,
		\item [(iii)] there are at most $\gamma n^2$ missing edges between $A_{1}$ and $A_2$.
	\end{itemize} 
	\item[(b)] Every $X \subset V(G)$ with $|X|=(1-\alpha)n$ satisfies $e_G(X) \geq \gamma^2 n^2$.
    \end{enumerate}
\end{lemma}

We will refer to the two cases above as, respectively, \emph{extremal} and \emph{non-extremal}. The following lemma resolves the latter case, and is the most difficult part of the proof. Note that here a smaller edge probability $p = n^{-2/s+o(1)}<p_s$ suffices. 

\begin{lemma}\label{l:non-extremal}
    Let $\gamma >0$, and let $G$ be a graph with minimum degree at least $(1-\frac sr)n$ in which every set of size $sn/r$  contains at least $\gamma n^2$ edges. For sufficiently large $C$ and $p = Cn^{-2/s}(\log n)^{2/(s(s-1))}$, w.h.p.\,the graph $G^p = G \cup \Gnp$ has a $K_r$-factor.
\end{lemma}

In the extremal case, we will construct the factor with copies of $K_r$ going across the partition $A_1 \dcup A_2$, 
most of which will contain $s$ vertices in $A_1$ and $t=r-s$ vertices in $A_2$. However, when $|A_1| > \frac{sn}{r}$ we will additionally need to find some copies of $K_{s+1}$ in $G^p[A_1]$, which is stated in the following lemma.
Note that this lemma is actually the bottleneck for the threshold probability $p_s$ when $s \geq 3$, as indicated by the extremal construction in Section~\ref{sec:lower-bound}.

\begin{lemma}\label{l:5.5chr}
    For all $s\geq 2$, there are constants $\beta, C>0$ such that the following holds. Let $G$ be an $n$-vertex graph with minimum degree $g\leq \beta n$. %  and let $Z \subset V(G)$ be a set of order at most (...)
    W.h.p.\,the graph $G^p = G \cup \Gnp$, with $p = Cp_s(n)$, contains a collection of $g$ vertex-disjoint copies of $K_{s+1}$. % such that each of these cliques contains at most one vertex of $Z$.
\end{lemma}

%For the proof (in general case), see Christian's notes, Lemma 4.3.

Another important ingredient is a  lemma which enables us to find $K_r$-factors in randomly perturbed super-regular pairs. In the present paper, the Lemma is only used for $m=1$.

\begin{lemma}\label{c:st-reg-pairs-factors}
    Let $r=ms+t$ with $t \in [s]$. For all $\vartheta \gg \eps \gg C^{-1}$, the following holds. Let $G$ be an $(m+1)$-partite graph with vertex classes $V_1 \dcup \dots \dcup V_{m+1} = V$ satisfying 
    $$|V_i| = \frac{s|V|}{r} \ \text{ for } i \in [m] \text{ \quad and \quad } |V_{m+1}| = \frac{t|V|}{r}.$$ Suppose that each pair $(V_i, V_j)$, $i \neq j$, is $( \eps, d, \vartheta)$-super-regular with $d \geq \vartheta$. For $p = C{n^{-2/s}}\left(\log n \right)^{2/(s(s-1))}$, if $n=|V|$ is divisible by $r$, then w.h.p.\,the graph $G \cup \Gnp$ contains a $K_r$-factor.
\end{lemma}

Now we can prove the main theorem.
\begin{proof}
	[Proof of Theorem~\ref{t:main}]
	Let $\beta \ll 1/s$. Let $\gamma$ be sufficiently small for Lemma~\ref{c:st-reg-pairs-factors} to hold with $\eps = \sqrt{\gamma}$ and 
    \[\gamma \leq \min \{\beta_{\ref{l:5.5chr}}, \beta^2 / (8s^2(s+1))\}.\] 
    Moreover, let $C \geq \max \{C_{\ref{l:non-extremal}}, C_{\ref{l:5.5chr}}, C_{\ref{c:st-reg-pairs-factors}} \}$ be sufficiently large.
	
	We assume that (*) any subset of $\Gnp$ of size at least $\frac{n}{\sqrt{\log n}}$ contains a copy of $K_s$, which occurs with high probability by Lemma~\ref{l:ks-in-lin-sized} (or, for $s=2$, by classic results on the independence number of $\Gnp$ found for instance in~\cite{jlr00}).
	
	Apply Lemma~\ref{l:4.3chr-simplified} with $\alpha = 1 - \frac sr$. In Case~(b), every $X \subset V(G)$ with $|X| = \frac{sn}{r}$ contains at least $\gamma^2 n^2$ edges of $G$, so we can find a $K_r$-factor using Lemma~\ref{l:non-extremal}.
	
	So assume Case~(a),  and let $g = |A_1| - \frac{sn}{r}$, with $|g| \leq \gamma n$.  Let $\sigma = g / |g|$ be the sign of $g$.
    \begin{claim}
        With high probability, $G^p$ contains a collection of $g$ vertex-disjoint copies of $K_r$, each using exactly $s+\sigma$ vertices from $A_1$ and $r-s-\sigma$ vertices in $A_2$.
    \end{claim}
    \begin{proof}
    \textbf{Case 1.} If $g\geq 0$, then by assumption of the theorem, $G[A_1]$ has minimum degree at least $g$. Hence, by Lemma~\ref{l:5.5chr}, we can find a collection $\cC_g$ of $g$ vertex-disjoint copies of $K_{s+1}$ in $G^p[A_1]$.
    
    \textbf{Case 2.} If $g<0$ and $r < 2s$ (so that $r-s < s$), let $\cC_g$ consist of $|g|$ vertex-disjoint copies of $K_{s-1}$ in $G^p[A_1]$, which can be found greedily using (*).
    
    In both cases, for each $K \in \cC_g$, recall that $N_G(K)$ denotes the set of common neighbours in graph $G$ of the vertices in $K$. The condition (i) from Lemma~\ref{l:4.3chr-simplified} gives us that
	$$|N_G(K) \cap A_2| \geq |A_2| - 4(s-1)\beta n \geq \frac{|A_2|}{2}.$$
	For each $K \in \cC_g$, using (*), we can find a copy of $K_{r-s-\sigma}$ in $\Gnp[N_G(K) \cap A_2]$; this is possible since {$r-s-\sigma \leq r-s+1 \leq s$}. Moreover, since $|g|s \leq \gamma s n \leq |A_2|/4$, we can make them vertex-disjoint by finding them one by one, each time deleting from $A_2$ the vertices already used.

    \textbf{Case 3.} Now assume that $r = 2s$ and $|A_1 | = n/2 + g$ for some $g <0$. Our aim is to find a collection $\cC_g$ of $|g|$ vertex-disjoint copies of $K_{s+1}$ in $G^p[A_2]$, each of them having at least $\beta n/2$ common neighbours in $A_1$. Recall that $G^p[A_2]$ has minimum degree at least $|g|$.

    Let $Z = \{ v \in A_2: |A_1 \setminus N_G(v)| \geq \beta n/(2s)\}$. Note that $|Z|\leq 2s/\beta \cdot \gamma n^2$, since otherwise there are too many missing edges between $A_1$ and $A_2$. Denote $Z = \{v	_1, \ldots, v_\ell \}$. If $\ell \geq |g|$, then we may greedily find $|g|$ vertex-disjoint copies of $K_{s+1}$ using property (*), where each copy uses exactly one vertex from $Z$. To see this, note that at the point of processing a~vertex $z \in Z$, its number of neighbours in $G[A_2]$ outside of $Z$ and the previously chosen cliques is at least
    $$\frac{\beta n}{2s}- |Z|(s+1)\geq n\left(\frac{\beta}{2s} - \frac{2s(s+1)\gamma}{\beta} \right) \geq \frac{\beta n}{4s},$$ so (*) can indeed be applied.
    If $\ell < |g|$, we first apply Lemma~\ref{l:5.5chr} to the graph $G[A_2 \setminus Z]$ to obtain $|g|-\ell$ copies of $K_{s+1}$ disjoint from $Z$, call them $\cC$. Then we use the same greedy argument as in the case $\ell \geq g$  (noting that $|V(\cC)| \leq |g|(s+1) \leq \gamma (s+1)n\leq\beta n/(8s)$) to complete the collection $\cC$ to $\cC_g$.

    For each copy of $K_{s+1}$ using at most one vertex from $Z$, the number of neighbours in $A_1$ is at least $\beta n - s \cdot \beta n /(2s) > \beta n / 2$. Now, as in the previous case, we can use property (*) to extend $\cC_g$ to a $K_r$-factor in $G^p[A_1 \cup A_2]$, since the total number of vertices in $A_1$ needed for the extension is $|g|(s-1) \leq \beta n/(8s).$
    
\end{proof}
    
    Let $\cC^+$ be the resulting collection of $|g|\leq \gamma n$ vertex-disjoint copies of $K_r$.
	
	Let $A_i' = A_i \setminus V(\cC^+)$ for $i \in \{1, 2\}$, and $G' = G[A_1' \cup A_2']$. There are still at most $\gamma n^2$ missing edges (in $G'$) between $A_1'$ and $A_2'$, and, since $\gamma \ll \beta$, any vertex has at least $\frac{\beta n}{2}$ neighbours in the opposite part. In particular, $(A_1', A_2')$ is a $(\sqrt{\gamma}, 1, \beta/2)$-super-regular pair. 
	Note that $n' = \frac{|A_1' \cup A_2'|}{r}$ is an integer, and, since $g = |A_1| - \frac{sn}r$, we have
	\begin{align*}
		t|A_1'| - s|A_2'| & = t|A_1| - s|A_2| - t|g|(s + \sigma) + s|g|(t-\sigma)
		\\
		& = t \left(\frac{sn}{r} + g - |g|(s + \sigma) \right) -
		s \left(\frac{tn}{r} - g - |g|(t - \sigma) \right)
		= 0,
	\end{align*}
	so $|A_1'| = sn'$ and $|A_2'| = tn'$. Hence, we can apply Lemma~\ref{c:st-reg-pairs-factors} to find a $K_r$-factor in $G' \cup \Gnp[A_1' \cup A_2']$, as required.
\end{proof}

\section{The extremal vs. non-extremal case distinction}\label{sec:cases}

In this section we prove Lemma~\ref{l:4.3chr-simplified}, which distinguishes the two cases for the deterministic graph $G$.
    
\begin{proof}[Proof of Lemma~\ref{l:4.3chr-simplified}]
	Suppose that (b) does not hold and let $X\subset V(G)$ be such that $|X|=(1-\alpha)n$ and $e_G(X) < \gamma^2 n^2$. If (i), (ii), and (iii) hold with $A_1 = X$ then we are done. Hence suppose this is not the case and let $Y=V(G)\setminus X$. Let $U\subseteq X$ consist of all vertices $x$ for which $|Y\setminus N_G(x)| > 2\beta n$. Suppose that $|U| \geq \gamma n/8$. Then, since $\delta(G)\geq \alpha n$,  $|Y|=\alpha n$ and $\beta > 8 \gamma$, we get
	\begin{align*}
		e_G(X) \geq |U|\cdot \beta n \geq \frac{\beta\gamma n^2}{8} > \gamma^2 n^2,
	\end{align*}
	a contradiction. Hence 
	\begin{align*}
		\xi := \frac{|U|}{n} < \gamma/8.    
	\end{align*} 
	Set
	\begin{align*}
		X_1 := X \setminus U \ \ \text{ and } \ \ Y_1 := Y \cup U,
	\end{align*}
	with
	\begin{align*}
		|X_1| = (1-\alpha-\xi)n \ \ \text{ and } \ \ |Y_1| = (\alpha + \xi)n,
	\end{align*}
	and note that $X_1$ satisfies (i) as for any $x\in X_1$ we have
	\begin{align*}
		|Y_1 \setminus N_G(x)| \leq 2\beta n + |U| \leq 4\beta n.
	\end{align*}
	Next, let $W\subset Y_1$ consists of all vertices $x$ for which $|N_G(x) \cap X_1| < \beta n$. We will estimate the number of edges between $X_1$ and $Y_1$. On one hand, since $e_G(X_1) \leq e_G(X) \leq \gamma^2n^2$, we have
	\begin{align*}
		e_G(X_1, Y_1) \geq |X_1|\alpha n - 2\gamma^2 n^2 = (1-\alpha-\xi)\alpha n^2 - 2\gamma^2n^2,
	\end{align*}
	and on the other hand
	\begin{align*}
		e_G(X_1, Y_1) & \leq |W|\beta n + |Y_1\setminus W||X_1| = |W|(\beta n - |X_1|) + |X_1||Y_1| \\
		& = -|W|(1-\alpha-\xi-\beta)n + (1-\alpha-\xi)(\alpha+\xi)n^2.
	\end{align*}
	Hence,
	\begin{align*}
		|W| \leq \frac{1}{1-\alpha-\xi-\beta}\left((1-\alpha-\xi)\xi + 2\gamma^2 \right) n \leq \left(2\xi + \frac{\gamma}{4}\right)n \leq \frac{\gamma n}{2},
		%|W| \leq \frac{1}{1-\alpha-\xi-\beta}\left((1-\alpha-\xi)\xi + 2\gamma_1^2 \right) n \leq \max\{2\xi, \gamma_1\} n,
	\end{align*}
	using the fact that the function $\frac{1- \alpha - \xi}{1-\alpha - \xi -\beta}$ is increasing in $\xi$, and $ 32\xi < \beta< \frac {1}{5}(1-\alpha)$.
	Set
	\begin{align*}
		X_2 := X_1 \cup W \ \ \text{ and } \ \ Y_2 := Y_1 \setminus W.
	\end{align*}
	We claim that $X_2 \dcup Y_2$ is the desired partition for (a) with $X_2$ playing the role of $A_1$. First notice that 
	\begin{align*}
		\left(1 - \alpha - \gamma\right)n \leq |X_2| \leq (1-\alpha- \xi)n + \frac{\gamma n}{2} \leq (1-\alpha+\gamma)n.
	\end{align*}
	Next, it is easy to see that $Y_2$ satisfies (ii). Moreover, $X_2$ satisfies (i) as for vertices in $X_1$ shifting vertices from $W$ to $X_1$ cannot increase the number of non-neighbours on the other side, and for any $x\in W$ we have
	\begin{align*}
		|Y_2 \cap N_G(x)| \geq (\alpha - \beta)n -|W|,
	\end{align*}
	hence
	\begin{align*}
		|Y_2\setminus N_G(x)| \leq |Y_2| - (\alpha-\beta)n + |W| = |Y_1| - (\alpha-\beta)n = (\xi + \beta)n \leq 2\beta n.
	\end{align*}
	Finally, to show (iii) notice that since $e_G(X) < \gamma^2n^2$, there are at most $2\gamma^2 n^2$ missing edges between $X$ and $Y$. When we shift $U$ from $X$ to $Y$ the number of missing edges between two parts can increase by at most $|U||X\setminus U| = \xi(1-\alpha-\xi)n^2 \leq \gamma(1-\alpha)n^2/8$, and when we shift $W$ from $Y_1$ to $X_1$ we can gain at most $|W||Y_1\setminus W| \leq \gamma(\alpha+\xi-\gamma/2)n^2/2$ additional missing edges between two parts. Hence in total we get at most 
    \[2\gamma^2 n^2 + \frac{\gamma(1-\alpha)n^2}{8} + \frac{\gamma(2\alpha + 2\xi - \gamma)n^2}{4} \leq \gamma n^2 \left(\frac{1}{20}+ \frac18 + \frac12 \right)\leq \gamma n^2\] 
    missing edges between $X_2$ and $Y_2$.
\end{proof}

%********************** Many disjoint cliques *************************************

\section{Many disjoint cliques}
    \label{sec:larger-clques}

In this section we prove Lemma~\ref{l:5.5chr}. Recall that for specific constants $\beta$ and $C$ our task is to find a collection of $g$ vertex-disjoint copies of $K_{s+1}$ in a randomly perturbed graph $G^p = G \cup \Gnp$ with 
\[\delta(G) = g \leq \beta n \ \ \text{ and } \ \  p = Cp_s(n).\]
Throughout the proof we set $V = V(G)$.

\begin{proof}[Proof of Lemma~\ref{l:5.5chr}]
    For $s=2$, the conclusion follows from Theorem 1.4 in~\cite{bpss23}. (We remark that our argument also applies to this case, but the calculations would have to be redone for $s=2$).

	Let  $s \geq 3$ and choose constants $C$, $\beta$ and $\delta$ so that
    \[ \beta \ll \delta \ll C^{-2s^2}. \] %we need $\beta \leq (100s(s+1))^{-1}$
    With high probability $\Gnp$ contains at least $\log^2 n$ vertex-disjoint copies of $K_{s+1}$ (which can be shown similarly as Lemma~\ref{l:ks-in-lin-sized}). Thus we may assume that $g \geq \log^2 n$. %Moreover, assume that $\beta \leq n/(2s)$.
	
	We will first show that we may also assume that the maximum degree of our graph is bounded from above by $\delta n$.

	\begin{claim}           
		Assume that the conclusion holds in case  $\Delta(G) \leq \delta n.$ Then it holds in general.
	\end{claim}
	
	\begin{proof}
		Let	$X = \{ x \in V(G): d_G(x) > \delta n / 2\}$, and let $X_1 \subset X$ be a set satisfying $|X_1| = \min(|X|, g)$. Denote $g_1 = g - |X_1|$. We claim that the graph $G^p[V \setminus X_1]$ contains $g_1$ vertex-disjoint copies of $K_{s+1}$. Indeed, when $|X_1|=g$ then $g_1 = 0$ and there is nothing to show. Hence suppose that $|X_1| = |X| < g$. Then the graph $G - X_1$ satisfies the hypothesis of the Claim with $g_1$ instead of $g$. Indeed, since $|X_1| < g \leq \beta n$, we have $n_* := |V \setminus X_1| \geq n/2$. Moreover, the minimum degree of $G-X_1$ is at least $g -|X_1| = g_1$, and for the maximum degree of $G - X_1$ we have 
        \[\Delta(G-X_1) \leq \delta n / 2 \leq \delta n_*.\]
        Thus $G^p[V \setminus X_1]$ contains $g_1$ vertex-disjoint copies of $K_{s+1}$.
		
		Now, for every $x \in X_1$, we may greedily choose a further copy of $K_{s+1}$ containing $x$, %and avoiding $Z$, 
        as follows. When processing $x$, at most $(s+1)g  \leq (s+1)\beta n \leq \delta n / 10$ vertices from $V \setminus X_1$ have already been used. Hence $N_G(x) \setminus X_1$ %$N(x)  \setminus X_1 \setminus Z$ 
		has at least $\frac {\delta n}{2} - g - \frac{\delta n}{10} \geq \delta n/4$ unused vertices; by Lemma~\ref{l:ks-in-lin-sized} %(and here is where we need $\delta \geq 1 / \log n$), 
        this set contains a copy of $K_s$ in $\Gnp$, which forms the desired copy of $K_{s+1}$ together with $x$.
	\end{proof}
	
	\begin{claim}
		\label{c:beta-const}
		The conclusion holds if $g \geq n / (10s\log n)$.
	\end{claim}
	
	\begin{proof}
		Let  $g \geq n / (10s\log n)$, and we may assume that $\Delta(G) \leq 10(s+1)g$.          
		Similarly to the argument above, we will greedily choose vertex-disjoint $K_{s+1}$-copies in $G \cup \Gnp$ as follows. Suppose that we have chosen $j \leq g-1$ vertex-disjoint copies of $K_{s+1}$ which cover a vertex set $J$ with $|J| = (s+1)j$. Assuming $j\leq g-1 \leq \frac{n}{100s(s+1)}$, the graph $G - J$ %or G-Z- J$, and using $|Z|\leq \frac{n}{100}$,
		has at least
		$$\frac{ng}{2} - \Delta(G) (s+1)j  \geq \frac{ng}{4}$$
		edges, and thus contains a vertex $x$ of degree at least $\frac{g}{2} \geq \frac{n}{20s \log n}$. By the conclusion of Lemma~\ref{l:ks-in-lin-sized}, the set $N_G(x) \setminus J $ contains an $s$-clique in $\Gnp$, which forms a desired copy of $K_{s+1}$ together with $x$. 
	\end{proof}
	
	To finish the proof, we show that the conclusion holds for $\delta(G) = g < n/(10s \log n)$ and $\Delta(G) \leq \delta n$. %Denote $\delta = (\log n)^{-1}$ (In this case we will also be able to find, say, $10g$ copies of $K_{s+1}$).
	Take a partition $A \dcup  B \dcup D$ such that $|A| = |B| = |D| = \frac n3$, %$Z \subset B$
	and for $v \in A$, $N_G(v) \cap B \geq \frac g5$; a random partition has this property with positive probability. %since $|B \setminus Z|$ contains at least $\frac n4$ randomly chosen elements.
	Let $G'\subset G$ be a bipartite subgraph on $A \dcup B$ in which each $v \in A$ has exactly $\frac g5$ neighbours in $B$, so $e(G') = \frac{gn}{15}$.
	
	\begin{claim} \label{c:cherries}
		$G'$ has at most $\delta g n^2$ cherries (that is, triples $uvw$ with $uv, vw \in E(G')$).
	\end{claim}
	
	\begin{proof}
		There are at most $\frac n3 \left(\frac g5 \right)^2 \leq \frac{1}{75} \beta gn^2 \leq \frac{1}{75} \delta gn^2$ cherries with a centre in $A$. Furthermore, for cherries $uvw$ with a centre $v \in B$,  there are at most $\frac n3$ choices for $u \in A$, $\frac g5$ further choices for its neighbour $v$, and $\Delta(G')\leq \delta n$ choices for $w$, which yields a total upper bound of $\delta g n^2$.
	\end{proof}
	
	We will now construct an auxiliary graph $(W, F)$ consisting of potential copies of $K_{s+1}$ intersecting $A \cup B$ in a single edge of $G'$, that is,
	$$W=\left\{K \in {V \choose s+1}: K \setminus D \in E(G') \right\} \text{ \quad and \quad }
	F = \left\{KL \in {W \choose 2}: K \cap L \neq \emptyset\right\}.$$
	We have $|W| = \kappa g n^s$, where the constant $\kappa$ depends only on $s$. %Informally, we will consider the random subgraph of $(W,F)$ which consists of $K_{s+1}-copies$ present in $G' \cup \Gnp$. 
	Let $(\tilde W, \tilde F)$ be the random induced subgraph of $(W, F)$ defined by
	$$\tilde W = \left\{ K \in W: {K\choose 2}\setminus E(G') \subset \Gnp \right\}.$$
	We will show that w.h.p.
	\begin{equation}
		\label{eq:W-vert-edge}
		|\tilde W| \geq 15g  \text{\quad and \quad}  \left| \tilde F \right | \leq 3g,
	\end{equation}
	so we may delete one element of each pair in $\tilde F$ and obtain $10g$ independent vertices of $\tilde W$, corresponding to $10g$ vertex-disjoint copies of $K_{s+1}$ in $G' \cup \Gnp$.
	
	Using the fact that 
	\begin{equation} \label{eq:nps}
		n^sp_s^{\binom {s+1}{2}-1} = n^{s-\frac{2s}{(s-1)(s+2)}\cdot\frac{s^2+s-2}{2}} = 1,
	\end{equation}
	and denoting $C' = C^{\binom {s+1}{2}-1}$, we have
	\begin{equation}
		\E |\tilde W| = |W|p^{\binom {s+1}{2}-1} = \kappa g C' \geq 30g,
	\end{equation}
    provided $C=C(s)$ is large enough. 
    
    We will now show that w.h.p.
	\begin{equation} \label{eq:e-tilde-f}
		\E |\tilde F| \leq 2g,
	\end{equation}
	and later we will argue that this also suffices for an application of Janson's 
	Inequality to $\tilde W$.
	Let $F_{i,j} \subset F$, $1 \leq i \leq s$ and $0 \leq j \leq 2$, be the set of pairs $KL \in {W\choose 2}$ with $|K \cap L|=i$ and $|(K \cap L) \setminus D| = j$. Then
    \begin{align}\label{eq:e-F_ij}
        \E |F_{i,0}[\tilde W]| & \leq (gn)^2 \cdot n^{2s-i-2} \cdot p^{2{s+1\choose 2} - 2 - {i\choose 2}} \leq (C')^2 g^2 n^{-i}p^{-{i\choose 2}} \leq (C')^2 \delta g n^{1-i} p^{-{i \choose 2}}, \\
        \E |F_{i,1}[\tilde W]| & \leq \delta gn^2 \cdot n^{2s-i-1} \cdot p^{2{s+1\choose 2} - 2 - {i\choose 2}} \leq (C')^2 \delta g n^{1-i} p^{-{i \choose 2}}, \\
        \E |F_{i,2}[\tilde W]| & \leq gn \cdot n^{2s-i} \cdot p^{2{s+1\choose 2} - 1 - {i\choose 2}} \leq (C')^2 gn^{1-i} p^{1-{i\choose 2}},
    \end{align}
    where the middle inequality (case $j=1$) is obtained using Claim~\ref{c:cherries}.
    Putting 
    \[F_1 = F_{1,0} \cup F_{1,1}\]
    and
    \[F_{2} = \left(\bigcup_{2\leq i \leq s} F_{i,0} \cup F_{i,1} \cup F_{i,2}\right)\]
    we get
    \[ \E|\tilde F_1| = \E|F_1[\tilde W]| \leq g \text{\quad and \quad}  \E|\tilde F_{2}| = \E|F_{2}[\tilde W]| = o(g),\]
    provided $\delta$ is sufficiently small with respect to $C^{-1}$. Therefore~\eqref{eq:e-tilde-f} holds. 
	
	To apply Janson's inequality to $\tilde W$, note that
	$$\Delta := \frac12 \sum_{KL \in F, K\cap L \neq \emptyset} \pr{K \in \tilde W \land L \in \tilde W} \leq \E|\tilde F| \leq 2g.$$
	Recalling that $\E |\tilde W| \geq 30g$ for sufficiently large $C = C(s)$, by Janson's inequality we get that w.h.p.
	\begin{equation}
		|\tilde W| \geq 15 g,
	\end{equation} 
	which covers the first part of~\eqref{eq:W-vert-edge}. 
 
    To prove the second part, it is enough to show that w.h.p.\,we also have $|\tilde F_1| \leq 2g$. To this end we estimate $\Var{|\tilde F_1|}$. Let
    \[ I_{KL} = \begin{cases}
        1 & \text{if } KL \in \tilde F, \\
        0 & \text{otherwise.}
    \end{cases}\]
    Then
    \[ |\tilde F_1| = \sum_{|K\cap L| = 1} I_{KL}\]
    and
    \begin{align}\label{eq:var-F-tilde}
        \Var{|\tilde F_1|} & = \sum_{|K\cap L| = 1} \sum_{|K'\cap L'| = 1} \Cov{I_{KL}, I_{K'L'}. }
    \end{align}
    For pairs $KL$, $K'L'$ in $\tilde F$ which do not share a random edge, $\Cov{I_{KL}, I_{K'L'}} = 0$. Hence each non-zero term in the sum in~\eqref{eq:var-F-tilde} corresponds to pairs $KL$, $K'L'$ having at least 2 vertices in common. Moreover, similarly to \eqref{eq:e-F_ij}, this expression is dominated by quadruples with the number of common vertices as small as possible, hence 2, sharing exactly one random edge and no deterministic edges. Thus
    \begin{align*}
       \Var{|\tilde F_1|} & \leq (1 + o(1)) gn^4 n^{4(s-1)-1}p^{4{s+1\choose 2}-5} = (1+o(1))g(np)^{-1} = o(g)
    \end{align*}
    By Chebyshev's inequality, w.h.p.\,$|\tilde F_1| \leq 2g$ and thus also $|\tilde F| \leq 3g$, which finishes the proof.
\end{proof}

% ********************* SPREAD APPLICATIONS ******************************

\section{Factors in regular partitions}
\label{sec:spread-applications}

In this section we apply the powerful result from~\cite{fknp21}, which confirmed the fractional version of the famous Kahn--Kalai Conjecture. See Section~\ref{sec:prelim-kk} for the relevant definitions and theorems.

\subsection{Covering small vertex sets}
The following lemma will be important in the proof of Lemma~\ref{l:non-extremal}. Its proof is a simple application of Theorem~\ref{thm:fractional-kk}, although there probably are other arguments.
Before we proceed, let us define the graph $B_{m,s,t}$ to be the complete $(m+1)$-partite graph with $m$ classes of size $s$, and one class of size $t\leq s$. In particular, $|V(B_{m,s,t})| = ms + t = r$. Moreover, for a graph $G$, a vertex $x\in V(G)$ and a subset $Y \subset V(G)$, let $\cB_x$ be a collection of $B_{m,s,t}$-copies in $G$ containing $x$ in a class of size $t$ and let $\cB_x[V(G) \setminus Y]$ denote the members of $\cB_x$ disjoint from $Y$.

\begin{lemma}
	\label{l:x-cover}
	Let $\eps, \nu \in (0,1/2)$, $G$ be an $n$-vertex graph, and $X \subset V(G)$ with $\chi := |X| \leq \eps n$. Assume that for all $x \in X$ and $Y \subset V(G) \setminus \{ x\}$ of size at most $2r \eps n$, $|\cB_x[V(G) \setminus Y]| \geq \nu n^{ms+t-1}$. For sufficiently large $C=C(\nu)$, and $p = Cn^{-2/s} (\log n)^{2/(s(s-1))}$, the graph $G^p = G \cup \Gnp$ contains a family $\cF = \{\tilde F_x: x \in X\}$ of vertex-disjoint copies of $K_r$ such that for each $x\in X$, $ \tilde F_x$ contains an element of $\cB_x$.  (In particular, $x$ is covered by a partial $K_r$-factor $\cF$.) % some element of $\cB_x$ is contained in some element of $\cF$.%partial $K_r$-factor covering the vertices of $X$.%!!
\end{lemma}

\begin{proof}
	
	We assume that $t=s$, as otherwise we may add an auxiliary set of $n$ vertices $W$ fully connected to vertices in $G$, and consider $B_{m,s,t+(s-t)}$-copies with $s-t$ vertices in $W$. %there will be many of those for each x, then we apply the t=s - statement, we find a collection of K_{(m+1)s}-copies covering X, which reduces down to $B_{m,s,t}$-copies covering x.
	
	Our ground set will be $Z=K_n^{(s)}$, the complete $s$-uniform hypergraph. A hypergraph $H \subset Z$ is called \textit{useful} if $|H |=(m+1)|X|$ and there is a collection of vertex disjoint graphs $\{  F_x: x \in X\}$ with $ F_x$ in $\cB_x$ such that $\partial  H \cup  \bigcup_{x \in X} {F}_x$ is the desired collection of vertex-disjoint $K_r$-copies. Let $\cH$ be the collection of useful hypergraphs. For a graph $F$ isomorphic to $B_{m,s,s}$, we write $\kappa_s(F)$ for the $s$-uniform hypergraph consisting of $m+1$ independent sets of size $s$ in $F$. Recall that $Z_\pi$ is the random subhypergraph of $Z$ where each $s$-edge is included independently with probability $\pi$. 
	
	\begin{claim}
		With high probability, $Z_\pi$ with $\pi= Cn^{-(s-1)}\log n $ contains an element of $\cH$.
	\end{claim}
	
	\begin{proof}
		We will construct an $O(n^{-s+1})$-spread measure $\mu$ on $\cH$. Let $Y_0 = \emptyset$. For steps $j=1, \ldots, \chi$, we choose a random element $x_j$ of $X \setminus Y_{j-1}$, and a random element $ F_j$ of $\cB_{x_j}[V(G) \setminus Y_{j-1}]$.  Then we set $Y_j$ to be $Y_{j-1} \cup V(F_j)$.  % $Y_j \cup \{ x_j\} \cup V(F_j)$. 
        (In this case, we say that $V(F_j)$ or its subset is \textit{selected} in step $j$.)  After $\chi$ steps, we obtain a random element $H_\mu = \bigcup_{j \in [\chi]} \kappa_s(F_j)$ of $\cH$, and $\mu$ is the probability measure on $\cH$ defined by this process.
		%The measure $\mu$ returns the collection $H_\mu = \bigcup_{j \in [\chi]} \kappa_s(F_j) $.
        To see that the hypergraph $H_\mu$ is indeed in $\cH$ (i.e., that it is useful), note that $\partial  H_\mu \cup  \bigcup_{j \in [\chi]} {F}_j$ is a partial $K_r$-factor covering~$X$. (We remark that the randomisation over $X$ is done  to make the following spread computation more symmetric by avoiding the special r\^{o}le of $s$-sets intersecting $x$.)
		
		To verify that $\mu$ is a spread measure, consider $S \subseteq Z = K_n^{(s)}$. To bound the probability that $H_\mu  =  \bigcup_{j \in [\chi]} \kappa_s( {F}_j)$ contains $S$, suppose that this is the case, so for each element $e \in S$, there is an index $j =j(e)$ such that $\kappa_s(F_j)$ contains $e$, and let $\mathcal J$ be the set of all  indices $j(e)$ for $e \in S$. (We remark that $|\mathcal{J}| \leq |S|$, but the inequality is strict if $j(e) = j(e')$ for some distinct $e, e' \in S$.) The number of potential choices for the indices $j(e)$ and the configuration $(x_j, F_j): j \in \mathcal{J}$ is at most $\chi^{|\mathcal{J}|}n^{r|\mathcal{J}|-s|S|}$. %(for, each element of j(e) can be selected in one of the $|X|$ steps of the algorithm; then there are at most $n^{(r-1)|\mathcal{J}|-s|S|}$ ways to `fill in' the remaining vertices of the configuration. s|S| is the number of vertices fixed by the condition that $S$ is in H_\mu)
		On the other hand, each such configuration $(x_j, F_j): j \in \mathcal{J}$ occurs with probability at most $$\frac {1}{\chi(\chi-1) \dots (\chi-|\mathcal{J}|+1)}\cdot \left(\nu n^{r-1} \right)^{-|\mathcal{J}|} \leq \left(\frac{\nu}{2e} \cdot \chi n^{r-1}\right)^{-|\mathcal{J}|},$$
		where the first term stands for the probability that the elements of $X$ are placed `correctly', and the second term comes from $|\cB_{x_j} \setminus Y_{j-1}| \geq \nu n^{r-1}$ for all $j$.
		
		Putting these bounds together, we get that the random hypergraph $H_\mu$ contains $S$ with probability at most 
            $$\chi^{|\mathcal{J}|}n^{r|\mathcal{J}|-s|S|} \left(\frac{\nu}{2e} \cdot \chi n^{r-1}\right)^{-|\mathcal{J}|} \leq A^{|\mathcal {J}|} n^{|\mathcal{J}|-s|S|} \leq (An^{-s+1})|S|,$$
		where $A=2e/\nu$, as required.

		Now, Theorem~\ref{thm:fractional-kk} implies that if $\pi = C'n^{-s+1}\log n$ for large $C' = C'(\nu)$, $Z_{\pi}$ contains a~member of $\Hc$ with high probability. 
	\end{proof}
	
	For $s=2$, this completes the proof. For $s \geq 3$, Lemma~\ref{l:coupling} implies that, with high probability, a random graph $\Gnp$ with $ p = 2\pi^{2/(s(s-1))} \leq Cn^{-2/s} (\log n)^{-2/(s(s-1))}$ contains $\partial H$ for some $H \in \Hc$. By definition of $\Hc$, this graph $\partial H$ along with the edges of $G$ gives the desired partial $K_r$-factor $\cF$.

\end{proof}

\subsection{$K_r$-factors in randomly perturbed regular pairs}
We now prove the main result of this section. Our application closely follows the elegant proof of Theorem 1.9 from~\cite{psss22}. Our proof significantly differs from theirs in the fact that we are \textit{augmenting} the underlying graph, rather than subsampling as is done in~\cite{psss22}.

Recall that a pair of sets $(V_1, V_2)$ in a graph $G$ is $(\eps, d, \vartheta)$-super-regular if it is $(\eps, d)$-regular, and for $i \neq j$, every vertex $v \in V_i$ has at least $\vartheta|V_j|$ neighbours in $V_j$. Note that if this is the case with $d \geq \vartheta \geq 4\sqrt{\eps}$, then $G[V_1, V_2]$ contains a spanning subgraph $G'$ for which $G'[V_1, V_2]$ is $(4\sqrt {\eps}, \vartheta, \vartheta - 4\sqrt{\eps})$-super-regular (see~\cite[Lemma 2.12]{abcd22}). As usual, we refer to an $(\eps, d, d-\eps)$-super-regular pair as $(\eps, d)$-super-regular.

\begin{lemma}
	%this is lemma 7.2 from cfrag 4, but with no remainder and V_i replaced by V_{i,1}U... U V_{i,s}
	\label{l:regular-pairs-factors}
	For integers $m, s, r$  and constants $ \vartheta \gg \eps$, there exists a constant $C$ such that the following holds. Let $G$ be an $nr$-vertex graph and $(V_{ij})$ be a collection of $n$-vertex pairwise disjoint subsets of $G$, with $i \in [m]$ and $j \leq s^*(i) \leq s$. Suppose that for $i \neq i'$, each $(V_{ij}, V_{i'j'})$ is an $(\eps, d,  \vartheta)$-super-regular pair in $G$ with some $d \geq \vartheta$. For $r = s^*(1)+\dots+s^*(m)$ and $p = C{n^{-2/s}}\left(\log n \right)^{2/(s(s-1))}$, w.h.p.\,the graph $G \cup G_{nr, p}$ contains a $K_r$-factor.
\end{lemma}

We first show how the above lemma implies Lemma~\ref{c:st-reg-pairs-factors}, which is a similar, but simpler statement. 
%We say that a pair is $(\eps, d^+)$-super-regular if it is $(\eps, d_1, d_1-\eps)$-super-regular for some $d_1 \geq d$.

\begin{proof}[Proof of~\ref{c:st-reg-pairs-factors}]
	For $i\in[m]$, partition $V_i$ uniformly at random into $s$ parts $V_{i,1} \dcup \dots \dcup V_{i,s}$ of size $\frac{n}{r}$ each. %Technically, this can be done by assigning each vertex to  a part with probability $\frac{0.99}{s}$, and then splitting the remaining vertices arbitrarily to ensure size $\frac nr$.
 Similarly, partition $V_{m+1}$  uniformly at random into $t$ parts $V_{m+1,1} \dcup \dots \dcup V_{m+1,t}$ of size $\frac{n}{r}$.  Assume that for $i \neq i'$ and all appropriate $j, j'$, each vertex from $V_{i,j}$ has at least $\frac{\vartheta |V_{i',j'}|}{2}$ neighbours in $V_{i',j'}$ and vice versa -- using Chernoff-type bounds for the hypergeometric distribution (see, e.g.~\cite{skala2013hypergeometric}), this happens with high probability. This implies that each pair $(V_{i,j}, V_{i',j'})$ is $(2r\eps, d/2, \vartheta/2)$-super-regular. Hence, we can apply Lemma~\ref{l:regular-pairs-factors} with $n, m$ equal $n/r$, $m+1$, respectively, and with $s^*(1)=\dots=s^*(m-1)=s$, $s^*(m)=t$, to find a $K_r$-factor in $G \cup \Gnp$, as required.
\end{proof}

In the remainder of this section, we will prove Lemma~\ref{l:regular-pairs-factors}.	
We will apply Theorem~\ref{thm:fractional-kk} in a~more general setting, which is needed for the inductive construction of an $O(n^{-s+1})$-spread measure. Let $(V_{ij})$ be a collection of $n$-vertex pairwise-disjoint subsets of a graph $\tG$, with $i \in [m], j \in [s]$. Suppose that for $(i, j) \neq (i', j')$, $(V_{ij}, V_{i'j'})$ is an $(\eps, d, \vartheta)$-super-regular pair in $\tG$.			
Let 
\[\tG_i = \bigcup_{j < j'} \tG[V_{ij}, V_{ij'}].\]

\begin{remark}
	We emphasise that the bipartite graphs $\tG[V_{ij}, V_{i'j'}]$ play a different r{\^o}le when $i = i'$ compared to $i \neq i'$. Namely, for $i=i'$, they are a part of the ground set which will be subsampled with probability $p$, whereas for $i \neq i'$ they are a part of the fixed deterministic graph $G$. In our application, $\tG[V_{ij}, V_{ij'}]$ will actually be complete bipartite graphs for $j \neq j'$.
\end{remark}

Recall that $\tG_i^{(s)}$ denotes the collection of $K_s$-copies in $\tG_i$. Now, let the ground set be 
\[Z = \tG_1^{(s)}\dcup \dots \dcup \tG_m^{(s)}.\] 
For $i \in [m]$, let $H_i$ be a perfect $s$-uniform matching in $\tG_i^{(s)}$ (corresponding to the set of $s$-cliques of a $K_s$-factor in $\tG_i$). We say that the set $H=H_1 \dcup \dots \dcup H_m$ is \emph{extendable} if there is a $K_r$-factor in $\tG$ which contains $\partial H$. %(in other words, if for each $i$, we can order the cliques in $H_i$ in a way that yields a $K_r$-factor). 
Let 
\[\cH \subset 2^Z \text{ be the collection of extendable sets}.\] 

Note that $\cH$ may a priori be empty (although the Blow-up Lemma, or the argument of~\cite{psss22} implies that this is not the case, assuming $d \gg \eps$). However,  we will construct an $O(n^{-s+1})$-spread measure on $\cH$, which immediately implies that $|\cH| = \Omega(n^{s-1})$. 
This, combined with Theorem~\ref{thm:fractional-kk} and Lemma~\ref{l:coupling}, will imply the required statement.

Before constructing the measure, let us state the special case where $m=1$ and $s=2$. The construction from~\cite{psss22} is elegant and natural, so we describe it here and repeat a part of the proof.

Let $\Gamma = (A, B, E)$, $|A|=|B|$, be a bipartite graph such that the pair $(A, B)$ is $(\eps, d)$-super-regular with $\eps \ll d$. For each vertex $v \in A \cup B$, choose a uniformly random set of $C$ neighbours of $v$ (with repetitions), independently of all other vertices, and denote the resulting graph by $J_C = J_C(G)$. Condition on the event that $J_C$ has a perfect matching (which occurs with probability at least $3/4$ for sufficiently large $C$, as shown in~\cite{psss22}), and output an arbitrary matching $ M(J_C)$. This induces a probability distribution $\mu_C$ on the matchings in $\Gamma$. More formally, for a matching $M$ in $(A, B)$,
\begin{equation}
	\label{eq:two-spread}
	\mu_C(M) = \frac{\pr{M(J_C) = M}}{\pr{J_C\text{ has a perfect matching}}} \leq \frac 43 \pr{M(J_C) = M}.
\end{equation}

\begin{lemma}[\cite{psss22}]
	\label{l:two-spread}
	Let $d \gg \eps >0$, and let $(A, B)$ be an $( \eps, d)$-super-regular pair in $\Gamma$ with $|A| = |B| = n$. The distribution $\mu_C$ is a $(4Cn^{-1})$-spread  probability distribution supported on the perfect matchings in $(A, B)$.
\end{lemma}

\begin{proof}
	Denote the random variable representing the $\mu_C$-random matching by $W$. For any set of edges $S$ in $\Gamma$, we need to show that $\mu_C(W \supseteq S) \leq (C'n)^{-|S|}$,  so we may assume that $S \neq \emptyset$ is a partial matching. Now, for each edge $e \in S$, the probability that either endpoint of $e$ chooses $e$ in $J_C$ is at most $2C/n$; moreover, these events are mutually independent for all edges of $S$. Hence, using $|S| \geq 1$,
	$$\mu_C(W \supseteq S) \leq \frac 43 (2C/n)^{|S|} \leq (4C/n)^{|S|},$$
	as required.
\end{proof}

To inductively construct spread measures on regular pairs, we will need the following lemma and notation.
If $K$ is a clique in a graph $\Gamma$, we write $\Gamma / K$ for the graph obtained by \emph{collapsing} $K$ to a single vertex; that is, $\Gamma/K$ has vertex set $(V(\Gamma) \setminus V(K)) \cup \{ K\}$, all the edges of $\Gamma - V(K)$, and all edges $xK$ with $V(K) \subseteq N_{\Gamma}(x)$. Similarly, if $\cK$ is a collection of vertex-disjoint cliques in $\Gamma$, we can define $\Gamma / {\cK}$ by collapsing all cliques in $\cK$ one by one.
By convention, if $M$ is a perfect matching between sets $(V_{ij}, V_{ij'})$, $j < j'$, we will consider the vertex set of $\Gamma/M$ to be $V_{ij}$.

\begin{lemma}
	[Lemma 4.3 in \cite{psss22}]
	\label{l:collapse-regular}  %
	For $C^{-1}, A^{-1}, d \gg \eps>0$, there are $d'=d'(C, A, d)$ and $c'  = c'(C, A, d)$ such that the following holds. Let $\Gamma= (V_0, V_1, V_2)$ be an $(\eps, d)$-super-regular tripartite graph with $|V_1| = |V_2| = n$ and $|V_0| \leq An$. Assume that for each edge $v_1 v_2 \in \Gamma[V_1, V_2]$, there are at least $(d^2-\eps^{1/2})n$ vertices $v_0 \in V_0$ which are adjacent to both $v_1$ and $v_2$. Let $\mu$ be the distribution on the perfect matchings $M$ between $V_1$ and $V_2$ given by~\eqref{eq:two-spread}. Given a random matching $M$, with probability at least $1-e^{-c'n}$, $\Gamma/M$ contains a $(16C/\log(1/\eps)^{1/4}, d')$-super-regular subgraph.
\end{lemma}

Now we formally state the lemma on constructing the spread measure. Our proof closely follows that of~\cite{psss22}, but we need to modify it because of the fact that we are only sparsifying inside $\bigcup_j{V_{ij}}$, so our desired probability threshold (in the random 2-graph setting) is roughly $n^{-2/s}$.

\begin{lemma}
	\label{l:spread}
	For integers $m, s, r$  and constants $\vartheta \gg \eps$, there exists a constant $A$ such that the following holds. Let $(V_{ij})$ be a collection of $n$-vertex pairwise disjoint subsets of a graph $\tG$, with $i \in [m], j \in [s]$. Suppose that each $\tG[V_{ij}, V_{i'j'}]$ is an $(\eps, d, \vartheta)$-super-regular pair with $d \geq \vartheta$. For $r =ms$, let $Z$ and $\Hc$ be as defined above. Then there exists an $(An^{-(s-1)})$-spread measure on $\Hc$.
\end{lemma}

\begin{proof}
	We prove the statement by induction on $s$. For $s=1$ (and $r=m$), consider the `1-uniform matching' $H = \bigcup_{i \in [m]}V_{i1}$. By  Lemma~\ref{l:blow-up} (the Blow-up Lemma), $H$ is extendable, i.e.~there is a $K_r$-factor in $\tG$. So the output of our desired measure will be $H$ with probability 1, and this measure is a 1-spread.
	
	Now let $s>1$. We may assume that the pairs in $\tG$ are $(\eps, \vartheta, \vartheta-\eps)$-super-regular by passing to a subgraph of $\tG$ if necessary, using Lemma~\ref{l:super-reg-min}.
	We start by constructing a~distribution $\mu$ on the perfect matchings in  $\bigcup_{i\in[m]} \tG[V_{i,s-1}, V_{i,s}]$, such that the $r(s-1)$-partite graph obtained by contracting a $\mu$-random matching still consists of appropriately regular pairs with high probability. Let $d \gg d_1 \gg \dots \gg d_{m} \gg C^{-1} \gg \eps_m \gg \dots \gg \eps_1 \gg \eps$.
	
	Assume that for some $\ell < m$, we have already sampled random matchings $M_i$ on $\bigcup_{i\in[\ell]} \tG[V_{i,s-1}, V_{i,s}]$, and that each pair of $\tG / (M_1 \dcup \ldots \dcup M_\ell)$ % $\tG^-:= \tG / (M_1 \dcup \ldots \dcup M_\ell)$ 
    is $(\eps_\ell, d_\ell)$-super-regular. Recall that the vertex set of $\tG[V_{i,s-1}, V_{i,s}]/M_i$ is $V_{i,s-1}$. Let $F$ be the subgraph of $\tG[V_{\ell+1,s-1}, V_{\ell+1,s}]$ consisting of edges which have at least $(d_\ell^2-\eps_\ell^{1/2})n$ common neighbours in each of the remaining parts -- this graph is $(\eps_{\ell+1}, d_{\ell}/2)$-regular. Let $\mu_{C}$ be the distribution supported on perfect matchings in $F$ given by~\eqref{eq:two-spread} with the constant $C \ll d_{l+1}^{-1}$. Given a matching $M_{\ell+1}$ sampled from $\mu_{C}$ and $(i,j) \notin \{ (\ell+1, s-1), (\ell+1, s) \}$, let  $\tG^- = \tG / (M_1 \dcup \ldots \dcup M_{\ell+1})$ and $\Gamma_{ij,M_{\ell+1}} = \tG^-[V_{\ell+1,s-1},V_{ij}]$ be the `collapsed' bipartite graph. By Lemma~\ref{l:collapse-regular} and the union bound, with probability at least $1-e^{-c'n}$, each $\Gamma_{ij, M_{\ell+1}}$ has a subgraph $\Gamma'_{ij,M_{\ell+1}}$ which is $(\eps_{\ell+1}, d_{\ell+1})$-super-regular; if this occurs, we say that $M_{\ell+1}$ is \emph{fitting}. 
	Condition on the event that $M_{\ell+1}$ is fitting, and contract it. Repeating this procedure until $\ell+1 = m$, we have constructed a random matching $M = M_1 \dcup \ldots \dcup M_m$ which is fitting, that is consists of fitting matchings $M_i$, with probability $1-e^{-c''n}$. Let $\mu_1$ be its distribution.
	
	Contracting $M$, we obtain an $m(s-1)$-partite graph $\tG'$ which satisfies the hypothesis of the Lemma with $s-1, d_m, \eps_m$ in place of $s, d, \eps$. Let $\Hc' = \Hc(\tG')$ be the set of extendable $(s-1)$-uniform perfect matchings in $\tG'$, and let $\mu'$ be the $(A'n^{-s+2})$-spread distribution on $\Hc'$ given by the inductive hypothesis. Note that any $H_1' \dcup H_2' \dcup \dots \dcup H_m' \in \Hc'$ corresponds to an extendable set $H_1 \dcup H_2 \dcup \dots \dcup H_m$ in $\tG$ via `uncontracting' the matching $M$, so we may define the corresponding measure $\mu$ on $\Hc$. More formally,
	\begin{equation}
		\mu(H_1 \dcup \dots \dcup H_m) = \frac{\mu'(H_1/M_1 \dcup H_2/M_2 \dcup \dots \dcup H_m/M_m)}{\mu_1 \left[ M\text{ is fitting} \right]}.
	\end{equation}
	It remains to show that $\mu$ is an $(An^{-s+1})$--spread. 
	
	To this end, let $Y_1 \dcup \dots \dcup Y_m$ be an $s$-uniform hypergraph (with $Y_i \subset \tG_i^{(s)}$) with $y = \sum_{i \in [m]}|Y_i|$. We need to bound the probability that a set $H_1 \dcup \dots \dcup H_m \in \Hc$ sampled from $\mu$ contains $Y_1 \dcup \ldots \dcup Y_m$ (this probability is zero unless $Y_1 \dcup \dots \dcup Y_m$ are vertex-disjoint). Recall that $M$ is the random matching sampled from $\mu_1$. Let $Y$ be the matching on $\bigcup_{i \in [m]}\tG[V_{i,s-1}, V_{i,s}]$ induced by $(Y_1, \ldots, Y_m)$. If $H_i \supset Y_i$ for $i \in [m]$, then the matching $M$ contains $Y$, which occurs with probability at most $(4C/n)^{y}(1+o(1))$. %o(1) for the event that some M_i is not fitting
	Furthermore, for a fixed (fitting) matching $M \supseteq Y$, the $\mu'$-probability that each $H_i/M$ contains $Y_i/M$ is at most $\left(A'n^{-s+2} \right)^{y}$ by the inductive hypothesis. We conclude that
	$$\mu \left(H_i \supseteq Y_i \text{ for all }i \in [m] \right) \leq (An^{-s+1})^{y},$$
	for an appropriate constant $A$. Hence $\mu$ is indeed an $(An^{-s+1})$--spread, as required.
\end{proof}

We can now prove Lemma~\ref{l:regular-pairs-factors}. 
%Recall the notation for $G_i$, $G_i^{(s)}$ and $Z = G_1^{(s)}\dcup \dots \dcup G_m^{(s)}$. 
Let $G$ be as in the statement, and let $G_i$ be the complete graph on $\bigcup_{j \leq s^*(i)}V_{ij}$.
In fact, we will prove an even stronger statement, where $G(nr, p)$ is replaced by a $p$-random subgraph of $G_1 \dcup \dots \dcup G_m$, for some graphs $G_i$ such that $G_{i}[V_{ij}, V_{ij'}]$ is $(\eps, d, \vartheta)$-super-regular for each $i \in [m]$.

\begin{proof}[Proof of Lemma~\ref{l:regular-pairs-factors}]
	First assume that $s^*(i) = s$ for all $i \in [m]$, as otherwise we can add auxiliary sets $V_{ij}$, $s^*(i) < j \leq s$ connected to everything else. Let $Z = G_1^{(s)}\dcup \dots \dcup G_m^{(s)}$. % recall the notation for $G_i$, $G_i^{(s)}$ and $Z = G_1^{(s)\dcup \dots \dcup G_m^{(s)}}$.
	Lemma~\ref{l:spread} (applied with $\tG = G \cup \bigcup_{i \in [m]}G_i$) and  Theorem~\ref{thm:fractional-kk} imply that, for $\pi = O(n^{-s+1} \log n)$, the random $s$-uniform hypergraph $Z_\pi$ contains a copy of $\cH$.
	
	For $s=2$, this completes the proof. For $s \geq 3$, Lemma~\ref{l:coupling} implies that with high probability, a random subgraph of $G_1 \dcup \dots \dcup G_m$ with edge probability $ p = 2\pi^{2/(s(s-1))} \leq Cn^{-2/s} (\log n)^{2/(s(s-1))}$ contains $\partial H$ for some $H \in \Hc$. By definition of $\Hc$, this graph $\partial H$ along with the edges of $G$ gives the desired $K_r$-factor.
\end{proof}
%******EXTREMAL SECTION
%
%
%
%
%
%
%
\section{Auxiliary extremal results -- covering the \textit{reduced graph} by matchings}
\label{sec:6-chr}
%[This corresponds to Section 6 of Christian's notes.]
In this section we prove an auxiliary extremal result which will be used to cover an appropriate \textit{reduced graph} in the proof of Lemma~\ref{l:non-extremal}. Note that we will actually need a slightly stronger statement (Lemma~\ref{l:6.14chr}), which will be stated and proved later in this section.

 This result can be viewed as a stability version of (a special case of) tiling results due to Shokoufandeh and Zhao or K\"uhn and Osthus, which strengthens Koml\'os' Tiling Theorem~\cite{komlos00,ko09,sz03}. Another stability version of Koml\'os' theorem can be found  in~\cite{hhp19}, but for our proof,  the number of leftover vertices $C$ has to be independent of $\delta$. 
Recall that $r = s+t$ with $t \leq s$. Let $\Qone$ be the graph consisting of three disjoint vertex sets $L$, $M$, $N$ such that $L = s-t$, $|M|=|N|=t$, the edges between $L$ and $M$ form a complete bipartite graph, and the edges between $M$ and $N$ form a matching. Note that $|V(\Qone)| = s+t$. For a copy of $\Qone$, we often refer to the vertices representing $M$ as the \textit{$M$-part} of this copy, and similarly for $L$ and $N$ we refer to the $L$-part and $N$-part, respectively.

\begin{figure}
    \begin{tikzpicture}[scale=0.8]
        \centering
        \draw[fill=white] (0,0) ellipse (1cm and 2cm);
        \draw[fill=white] (4,0) ellipse (1cm and 2cm);
        \draw[fill=white] (8,0) ellipse (1cm and 2cm);
        \coordinate (v1) at (0,-0.5);
        \coordinate (v2) at (0,0.5);
        \coordinate (v3) at (4,-1);
        \coordinate (v4) at (4,0);
        \coordinate (v5) at (4,1);
        \coordinate (v6) at (8,-1);
        \coordinate (v7) at (8,0);
        \coordinate (v8) at (8,1);
        \begin{pgfonlayer}{front}
		\foreach \i in {v1,v2,v3,v4,v5,v6,v7,v8} \fill (\i) circle (3pt);	
        \foreach \i in {v3,v4,v5}
            \foreach \j in {v1,v2}
                \draw (\i) -- (\j);
        \draw (v3) -- (v6);
        \draw (v4) -- (v7);
        \draw (v5) -- (v8);
        \node at (0,-2.2) [below] {$L, |L|=s-t$};
        \node at (4,-2.2) [below] {$M, |M|=t$};
        \node at (8,-2.2) [below] {$N, |N|=t$};
    \end{pgfonlayer}

    \end{tikzpicture}
\caption{The graph $\Qone$ with $s=5$ and $t=3$}
\end{figure}
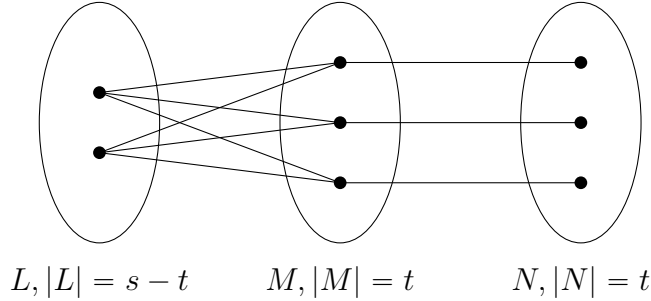

Note that for $s=t = r/2$, $\Qone$ is simply a matching of size $s$. The proof of Lemma~\ref{l:6.2chr} also holds in this case, but such a result can probably be found elsewhere in the literature. In the proof, we will apply the classical theorem of K\H{o}v\'ari, S\'os and Tur\'an, which states that for given integers $a$ and $b$, and $\epsilon>0$, any sufficiently large graph of density at least $\epsilon $ contains a copy of $K_{a,b}$~\cite{kst54}.

\begin{lemma}
	\label{l:6.2chr} There is a constant $C$ such that for $\delta \gg 1/n $ and $\gamma=10sr\delta$, if $G$ is an $n$-vertex graph with minimum degree $\delta(G) \geq \left(1- \frac sr - \delta \right)n$, then 
    \begin{itemize}
	   \item[(i)] $G$ contains vertex-disjoint copies of $K_{2}$ and $\Qone$ covering all but at most $C$ of its vertices, or
	   \item[(ii)] $G$ has an independent set of size $\left(\frac sr - \gamma \right)$n.
    \end{itemize}
\end{lemma}

\begin{proof}
    Let $\alpha = \frac tr$, so that $\delta(G) \geq (\alpha - \delta) n$. Fix constants $C \gg s, r$ and $\delta \gg n_0^{-1}$, and assume that $n \geq n_0$. Let $\Pc = \{K_1, K_2, \Qone \}$. Given a~$\Pc$-factor $\cF$ of $G$, we let $k_i^\cF$ be the number of $K_i$-copies in $\cF$ (for $i =1,2$) and $q^\cF$ be the number of $\Qone$ copies in $\cF$. 
	
	Evidently, $G$ has at least one $\Pc$-factor (e.g. each vertex can be covered by its own copy of $K_1$);  fix a factor $\cF$ of $G$ such that the vector
	$$(k_2^\cF + tq^\cF,  q^\cF)$$
	is lexicographically maximal. We call this vector the \emph{index of the factor}. 
	
	For $i \in [2]$, let $\fA_i$ be the set of $K_i$-copies in $\cF$, and let $A_i$ be the set of vertices defined by them. Define $\fB$ and $B$ similarly with respect to $\Qone$. If $|A_1| \leq C$, then the conclusion (i) holds. So suppose this is not the case.

	Note that $A_1$ is an independent set as otherwise we can take two elements of $\fA_1$ and replace them by a copy of $K_2$, which increases the index of the factor.
	
	\begin{claim}
		\label{claim:6.4chr}
		$e(A_1, A_2) \leq |A_1|(\alpha|A_2|/2+C)$.
	\end{claim}
	
	\begin{proof}
		Assume that this is not the case. Then  $|A_2| \geq C$ (as otherwise $e(A_1, A_2) \leq |A_1|C$) and
		\begin{equation}
			\label{eq:edges-a}
			e(A_1, A_2) >\alpha|A_1||A_2|/2.
		\end{equation}
		Let $O$ be the following orientation of the edges in $\fA_2$: for each $vw \in \fA_2$, include the ordered pair $(v, w) $ in $O$ if $|N_G(v)\cap A_1| \geq |N_G(w) \cap A_1|$; if $|N_G(v)\cap A_1| = |N_G(w) \cap A_2|$, the orientation of $vw$ is chosen arbitrarily. Construct an auxiliary bipartite graph $H = H(O)$ with vertex classes $\fA _1$ and $O$ by joining $u \in \fA_1$ to $(v, w) \in O$ if $u \in N_G(v)$. 
		Now, by the choice of $O$, the number of edges in $H$ is at least
		$$e(H) \geq \frac 12 e(A_1, A_2) > \frac 14 \alpha |A_1| |A_2|.$$
		Since $|\fA_1|\geq C$ and $|\fA_2| \geq C/2$, the theorem of K\H{o}v\'ari, S\'os and Tur\'an yields a copy of $K_{s-t, t}$ in $H$.
		Using this copy of $K_{s-t, t}$, we may merge its $s-t$ vertices and $t$ edges (of $\fA_2$) into a copy of $Q$, to create a new $\cP$-factor.
		
		This alteration increases $q^\cF$ without changing $k_2^\cF +tq^\cF$, %and $k_1 + (s-t)q$,
		contradicting the maximality of $\cF$.
	\end{proof}
 
	Given a copy $\tilde Q$ of $\Qone$, we use $\tilde L$, $\tilde M$ and $\tilde N$ to denote its $L$-part, $M$-part and $N$-part, respectively.
 
	\begin{claim}
		\label{claim:6.6chr}
		If $x \in A_1$ and $\tilde Q \in \fB$, then $x$ has no neighbours in $\tilde L \cup \tilde N$. In particular,
		$$|N_G(x) \cap V(\tilde Q) | \leq t = \alpha |V(\tilde Q)|.$$
	\end{claim}
 
	\begin{proof}
		Suppose that the claim doesn't hold. We will modify $\cF$ by removing $\tilde{Q}$ and adding $(t+1)$ copies of $K_2$, which gives a factor with a higher index, contradicting the choice of $\cF$. We have two cases for the vertex $u \in N_G(x) \cap \left( \tilde L \cup \tilde N \right)$. If $u \in \tilde L$, we simply add the matching between $\tilde M$ and $\tilde N$, as well as the edge $xu$ to $\cF$. If $u \in \tilde N$, let $u'$ be its $\tilde Q$-neighbour in $\tilde M$, and let $\ell \in \tilde L$. In this case, we add the edges $xu$, $u' \ell$ and $\tilde Q\left[\tilde M \setminus\{u'\}, \tilde N \setminus \{u\} \right]$ to $\cF$.
	\end{proof}
 
	\begin{claim}
		We have $|A_1 \cup A_2| \leq 4\delta rn$.
	\end{claim}
 
	\begin{proof}
		Using the minimum degree of $G$ and the fact that $A_1$ is an independent set, we have $e_G(A_1, V(G) \setminus A_1) \geq (\alpha-\delta)|A_1|n$. On the other hand, using the previous two claims,
		$$e_G(A_1, V(G) \setminus A_1) \leq |A_1|(\alpha|A_2|/2+C) + \alpha |A_1| |B|.$$
		Putting these bounds together, we get
		\begin{align*}
			(\alpha - \delta) n &\leq \alpha (n - |A_1|) - \frac{\alpha|A_2|}{2}+C,
		\end{align*}
        which in turn implies
        \begin{align*}
            \frac{1}{2r}(|A_1| + |A_2|) \leq \alpha |A_1| + \frac{\alpha|A_2|}{2} & \leq \delta n + C \leq 2 \delta n,
        \end{align*}
		as required.
	\end{proof}
	
	For a vertex $x \in A_1$, let $\mathfrak{D}_x$ be the family of $\Qone$-copies $\tilde Q$ whose $M$-set is a subset of $N_G(x)$. %, and let $D_x$ be the vertex set of $\mathfrak{D}_x$. 
    Now pick two distinct vertices $x$ and $y$ in $A_1$, and let $\mathfrak{D} = \mathfrak{D}_x \cap \mathfrak{D}_y$. Let $Z$ be the union of $L$-sets and $N$-sets of the copies of $Q$ in $\mathfrak{D}$. We will show that $Z$ is an independent set of size close to $(1-\alpha)n$.
	
	Firstly, assume that there is an edge $x_1 y_1$ in $G[Z]$. We modify $\cF$ as follows. Replace $x_1$ by $x$ and $y_1$ by $y$ in their respective $\Qone$-copies (noting that $x_1$ and $y_1$ may lie in the same copy). Furthermore, add the edge $x_1y_1$ (viewed as a copy of $K_2$) to $\cF$ -- this increases the index $k_2^\cF + tq^\cF$, contradicting the maximality of $\cF$. Hence, $Z$ is an independent set.
	
	Finally, we claim that $|Z| \geq \left(\frac sr - \gamma \right)n$ with $\gamma = 10sr \delta$, which yields conclusion (ii) and completes the proof of the Lemma. Firstly, to bound $|\fD_x|$, note that by the definition of $\fD_x$ and Claim~\ref{claim:6.6chr}
	$$\left(\frac tr -\delta \right) n \leq |N_G(x)| \leq |A_1 \cup A_2| +  t|\fD _x | + (t-1)\left(\frac nr - |\fD _x| \right) = |A_1 \cup A_2|+ |\fD_x| + \frac{(t-1)n}{r}.$$
	Rearranging, we get
	$$\frac nr-|\fD_x| \leq \delta n + |A_1\cup A_2| \leq 5\delta rn,$$
	and hence $|\fD_x| \geq \frac nr (1- 5r^2\delta )$. It follows that $|\fD| = |\fD_x \cap \fD_y| \geq \frac nr (1- 10r^2\delta )$.
	Since $Z$ is the union of the $L$-sets and $N$-sets in $\mathfrak{D}$, we have $|Z| = s|\fD_x| \geq \frac{sn}{r}(1-10r^2\delta ) \geq n \left(\frac{s}{r}-10sr \delta \right)$, as required.
\end{proof}

Now we turn to the stronger version of the previous lemma, in which we also require some \textit{absorbing} copies of $K_2$ in the constructed covering. Note that in most cases, when $r<2s$, we only need the following definition with $g=2$.
Let $\delta >0$ and let $g$ be a positive integer. A \emph{$(\delta, g)$-absorber} $\absorber$ in an $n$-vertex graph $G$ is a collection of copies of $K_g$ in $G$ covering at most $\delta n$ vertices such that for every $x \in V(G)$, there are at least $\frac{\delta^2}{24g^2} n$ cliques in $\absorber$ containing a  neighbour of $x$. 

\begin{lemma}
	\label{l:6.13chr}
	For $1 \gg \delta \gg 1/n \gg 0$, if $G$ is an $n$-vertex graph such that each vertex is in at least $\delta n^{g-1}$ copies of $K_g$, then $G$ contains a  $(\delta, g)$-absorber.
\end{lemma}

\begin{proof}
    For each vertex $v\in V(G)$ choose a family $\cK_v$ of $\delta n^{g-1}$ copies of $K_g$ containing $v$. Let $\cK = \bigcup_{v}\cK_v$ and $\alpha$ be such that $|\cK|=\alpha n^g$, with $\frac{\delta}{g}\leq\alpha\leq \delta$. Next, let $\cA'$ be a random subfamily of $\cK$ containing each element of $\cK$ independently with probability $\frac{\delta}{2\alpha g}n^{1-g} \leq \frac12 n^{1-g}$. Then $ \E(|\cA'|) = \frac{\delta}{2g}n$ and, by Chebyshev's inequality, with probability $1-o(1)$ we have $\frac{\delta}{4g}n \leq |\cA'| \leq \frac{\delta}{g}n$, and the copies of $K_g$ in $\cA'$ cover at most $\delta n$ vertices. 
    
    Next, notice that, since each vertex belongs to at least $\delta n^{g-1}$ copies of $K_g$, we have $\delta(G)\geq \delta n$. Therefore, there are at least $\delta n \cdot \frac{\delta}{g}n^{g-1} = \frac{\delta^2}{g}n^g$ unique copies of $K_g$ containing at least one neighbour of $v$. Denote by $\cK_{N(v)}$ any such family of $K_g$-copies of size exactly $\frac{\delta^2}{g}n^g$, let $\cA'_{N(v)}=\cA' \cap \cK_{N(v)}$ and note that $\cA'_{N(v)}$ is binomially distributed with parameters $\frac{\delta^2}{g}n^g$ and $\frac{\delta}{2\alpha g}n^{1-g}$. We have $\E(|\cA'_{N(v)}|) = \frac{\delta^3}{2\alpha g^2}n \geq \frac{\delta^2}{2g^2}n$. By Chernoff bounds, the probability that $\cA'_{N(v)} < \frac{\delta^2}{3g^2}n$ is at most $e^{-\beta n}$ for a proper constant $\beta=\beta(g,\delta)$. Thus, taking the union bound over all the vertices of $G$, with probability $1-o(1)$ for each $v\in V(G)$ there are at least $\frac{\delta^2}{3g^2}n$ copies of $K_g$ in $\cA'$ containing a neighbour of $v$.
    
    Next, we want to make $\cA'$ pairwise disjoint. Let $Z$ denote the number of intersecting pairs in $\cA'$ and note that $\E Z \leq n (\delta n^{g-1})^2 \left(\frac{\delta}{2\alpha g} n ^{1-g}\right)^2 = \frac{\delta^4}{4\alpha^2g^2}n\leq \frac{\delta^2}{4g^2}n$. Therefore, by Markov's inequality, with probability bounded away from 0, the number of intersecting pairs does not exceed $\frac{7\delta^2}{24g^2} n$. Let us remove one copy of $K_g$ from each intersecting pair and denote by $\cA$ the remaining family. Then, with positive probability, $\cA$ satisfies the definition of a $(\delta,g)$-absorber, hence we can find in $G$ the desired absorber.
\end{proof}

Let us now find the desired partial factor which includes an absorber. Note that in case $s>r/2$, we can find an independent set in (i), but the slightly weaker conclusion is necessary in the singular case.

\begin{lemma}
	%covering lemma with absorbers.
	\label{l:6.14chr} Let $s \geq r/2$. There is a constant $C$ such that for $\gamma \gg \delta \gg 1/n $, if $G$ is an $n$-vertex graph with minimum degree $\delta(G) \geq \left(1- \frac sr - \delta \right)n$, then 
    \begin{itemize}
	   \item[(i)] there is a subset $X \subset V(G)$ with $|X| \geq \left(\frac sr - \gamma \right)n$ containing at most $\delta n^2$ edges, or
    %$G$ has an independent set of size $\left(1- \frac sr - \gamma \right)n$, or
	   \item[(ii)] if $s > r/2$, then there is a collection $\mathcal{C}$ of vertex-disjoint copies of $K_{2}$ and $\Qone$ in $G$, covering all but at most $C$ of its vertices, such that $\mathcal{C}$ contains a $\left( \frac{\delta}{2}, 2\right)$-absorber $\absorber$, or
	   \item[(iii)] if $ s = r/2$, then there is a collection $\mathcal{C}$ of vertex-disjoint copies of $K_3$ and $K_2$ in $G$, covering all but at most $C$ of its vertices, such that  $\mathcal{C}$ contains a $\left(\frac \delta 2, 3\right)$-absorber $\absorber$.
    \end{itemize}
\end{lemma}

\begin{proof}
For (ii), find a $\left( \frac{\delta}{2}, 2\right)$-absorber using Lemma~\ref{l:6.13chr} and then apply Lemma~\ref{l:6.2chr} to the rest of the graph.

For (iii), assume that conclusion (i) does not hold. Then each vertex is in at least $\delta n^2$ triangles since its neighbourhood contains at least $\delta n^2 $ edges. Hence we can find a $\left(\frac \delta 2, 3\right)$-absorber $\absorber$ using  Lemma~\ref{l:6.13chr} and remove it from the graph. The remaining graph still satisfies the hypothesis of Lemma~\ref{l:6.2chr} with $\delta_{\ref{l:6.2chr}} = \delta/2$, so we may apply the Lemma, and recall that $\Qone$ is simply a matching of size $s$.
\end{proof}

When $t<s$, the graph $\Qone$ will be used in the reduced graph as an auxiliary object used to obtain a set of regular pairs $(S_i, T_i)$ with $|S_i|/|T_i|\sim s/t$. To formalise this, we will use the concept of weighted graphs (with weighted vertices) and their embeddings.

A \emph{weighted graph} is a pair $(F, w_F)$, where $F$ is a graph and $w_F: V(F) \to \R_{\geq 0}$ is a~weight function. Any graph $F$ can be viewed as a weighted graph with $w_F \equiv 1$. 
For a~family of weighted graphs $\cF$, an \emph{$\cF$-packing} in $G$ is a collection of graphs $(F_i)_{i \in I}$, where $F_i \in \cF$, with corresponding embeddings $(f_i)_{i \in I}$ into $G$, such that for $x \in V(G)$, 
$$\sum_{i \in I: Im(f_i) \ni x} w_{F_i}(f_i^{-1}(x)) \leq 1.$$
In other words, for each $x \in V(G)$ the total weight in $(F_i)_{i\in I}$ of all vertices mapped to $x$ is at most 1.

The \textit{residue} of a packing is defined as $|V(G)|- \sum_{i \in I}\sum_{v \in V(F_i)}w_{F_i}(v)$. Let $K_2^*$ denote the weighted graph consisting of two vertices $\sigma$ and $\tau$ with an edge $\sigma \tau$, and vertex weights $w(\sigma) = \frac{1}{t(s+t)}$ and $w(\tau) = \frac{1}{s(s+t)}$ (so that $w(\sigma)/w(\tau) = s/t$).

For example, the graph $K_2$ with vertices $x$ and $y$ (with vertex weights 1) has a $K_2^{*}$-packing consisting of $2st$ copies of $K_2^*$: $st$ copies $F_i$ with $f_i(\sigma) =x$ and $f_i(\tau) = y$, and $st$ copies $F_j$ with $f_j(\tau) = x  $ and $ f_j(\sigma) = y$. The residue of this packing is 0. We will use the following simple lemma establishing the same statement for $\Qone$.

\begin{lemma}
	\label{l:k2-packing}
	The graph $\Qone$ has a $K_2^*$-packing with residue 0 such that for each copy of $K_2^*$ in the packing, its vertex $\tau$ is mapped into the $M$-part of $\Qone$.
\end{lemma}

\begin{proof}
	We begin by enumerating the vertices of $Q$ as $L = \{\ell_1, \dots, \ell_{s-t}\}$, $M  = \{m_1, \ldots, m_t \}$ and $N = \{n_1, \ldots, n_t \}$ so that each $m_i$ is matched to $n_i$ in $Q$. Let $(s+t) \ltimes K_2^*$ be the weighted graph on $\{\sigma, \tau\}$ with weight function $w'(\sigma) = 1/t$ and $w'(\tau)  = 1/s$. We will actually exhibit an $(s+t) \ltimes K_2^*$-packing of $Q$, which can easily be turned into a $K_2^*$-packing by partitioning each $(s+t) \ltimes K_2^*$-copy into $(s+t)$ copies of $K_2^*$. % Let $(\sigma_i, \tau_i)_{i \in [st]}$ be the vertex sets of $(s+t)\ltimes K_2^*$-copies which will be packed into $Q$.
	
	First for each $j \in [t]$, we embed $t$ copies of $(s+t)\ltimes K_2^*$ by mapping $\sigma$ to $n_j$ and $\tau$ to $m_j$. Hence the total weight mapped to $n_j$ is 1, and to $m_j$ is $\frac{t}{s}$.
	
	Then, consider the vertex $\ell_1$. Embed $t$ copies of $(s+t)\ltimes K_2^*$ by mapping $\sigma$ to $\ell_1$ and $\tau$ to $m_j$, for $j \in [t]$. The total weight mapped to $\ell_1$ is now also 1. Repeat the same for $\ell_2, \ldots, \ell_{s-t}$.  
 
    Finally, notice that for $j\in[t]$, the total weight mapped to each vertex $m_j$ is
	$\frac {t}{s}  + \frac {s-t}{s} = 1.$ This completes the proof.
\end{proof}

\section{Graphs with no large almost independent sets}  	
\label{sec:non-extremal}

In this section we will prove Lemma~\ref{l:non-extremal}.

\subsection{The non-singular case $s>t$}

We start with an outline of the proof strategy. Recall the notions of the reduced graph and regular partitions from Section~\ref{sec:prelim-reg}.

\begin{description}
    \item[Step 1] Apply the Regularity lemma and form an $(\eps, \delta/10)$-reduced graph $R$ for appropriate $\eps, \delta$. %Some of the ready-to-use regularity tools can be found in~\cite{hmt21,abcd22,psss22,bpss23}.
    \item[Step 2] Recall the assumptions of Lemma~\ref{l:non-extremal}.  They imply that (for appropriate $\delta$), the reduced graph $R$ has minimum degree $(1-s/r - 2\delta)|V(R)|$ and no independent set of size $(1- s/r - \gamma)|V(R)|$.
    \item[Step 3] Use Lemma~\ref{l:6.14chr} and Lemma~\ref{l:k2-packing} to find a matching $\{(S_i, T_i): i \in [\ell]\}$ in the reduced graph $R$ which covers all but at most $C= C(s, t)$ vertices of $R$, and such that
    \begin{align}\label{eq:sizes}
        \frac{|S_i|}{|T_i|} =\frac st \text{ \quad for } i \in [\ell].
    \end{align}
    In other words, each $(S_i, T_i)$ is a dense super-regular pair, and we do not know anything about the rest of $R$.
\end{description}

Now we start introducing random edges. 

\begin{description}
    \item[Step 4] Cover $V_0$ with $K_r$-copies (in $G \cup \Gnp$).% while preserving~\eqref{eq:sizes}.
    \item[Step 5] Ensure that the parts are of required size for an application of Lemma~\ref{c:st-reg-pairs-factors}.
    \item[Step 6] Apply Lemma~\ref{c:st-reg-pairs-factors} to cover the remaining regular pairs.
\end{description}

\medskip
 The following lemma corresponds to steps 1--3 from the above outline. %In case $s = t=  r/2$, items (ii) and  (iii) still holds, but are irrelevant. %\nk{Change $\delta / 40$ to $\delta$.}

\begin{lemma}
	\label{l:st-reduced}
	Let $t$ and $s$ be positive integers with $t< s$, let $1/t \gg \gamma \gg d>0$, and let $\eps>0$ sufficiently small be given. There is $L = L(\gamma, \eps)$ such that the following holds for sufficiently large $n$. Let $G$ satisfy the assumption of Lemma~\ref{l:non-extremal}, that is $\delta(G) \geq \left(1-\frac {s}{r}\right)n$ and every vertex subset of size $sn/r$ contains at least $\gamma n^2$ edges in $G$. There is a collection $(S_i, T_i)_{i \in [\ell]}$ of $(\eps, d', d)$-super-regular pairs  with $\ell \leq L$, $d' \geq d$, and such that the following holds.
	\begin{enumerate} 
		\item \label{it:ST-sizes} For every $i \in [\ell]$, 
		$t|S_i| = s|T_i|$ and $|S_i \cup T_i| = \frac{n}{\ell}(1\pm \eps^{1/2})$.
		\item \label{it:S-neighbours} Every vertex $x \in V(G)$ has at least $dn$ neighbours in $\bigcup_{i \in [\ell]}S_i$.
		\item \label{it:S-reduced} For every set $S_i$, there are at least $d\ell$ sets $S'$ in $\{S_1, \ldots, S_\ell\}$ such that $(S_i, S')$ is an $(\eps,d', d)$-super-regular pair with density $d' \geq d$.
	\end{enumerate}
\end{lemma}

\begin{proof}%[Proof of Lemma~\ref{l:st-reduced}]
	Given $\gamma>0$, let $\delta \ll \gamma$ %$\delta \in (0, 1/100)$ 
    be sufficiently small for Lemma~\ref{l:6.14chr} to hold with $\gamma_{\ref{l:6.14chr}}=\gamma/4$, and let { $d_1 = \delta /10$}. We will show that the conclusion holds with $\sqrt{\eps}$ in place of $\eps$, and we may assume that $\eps \ll \min\{\delta^8, d_1^3\}$. Apply the Regularity Lemma (Lemma~\ref{l:regularity}) to $G$ with constants $\eps$, $d_1$ and $\ell_0 > \max\{8 \gamma^{-1},\eps^{-2}\}$. Let $V_0 \cup \dots \cup V_{\ell'}$ be the obtained $\eps$-regular partition, and $R$ its $(\eps, d_1)$-reduced graph. Then $R$ satisfies $\delta(R) \geq \left(1-\frac sr-2 \delta \right)\ell'$.
	
	We also claim that $R$ does not contain a vertex set $I$ of size at least $\left(\frac sr - \frac{\gamma}{4} \right)\ell'$ with $e_R(I) \leq  2\delta (\ell')^2$ edges. Assume the opposite, and let $I$ be such a set. Let $V_I$ be the union of the corresponding parts of the regular partition of $G$. We have
	$$e(G[V_I]) \leq \binom n2 (4 \delta + d_1 + 2\eps + \ell_0^{-1}) \leq \gamma n^2/4,$$
	since $G[V_I]$ can contain  at most $2 \delta n^2$ edges inherited from the edges of $R$, along with the edges in  pairs $(V_i, V_j)$ which are not $(\eps, d_1^+)$-regular, and edges inside the parts. Moreover, since $|V_i| \geq \frac {n}{ \ell'}(1- 2\eps)$ for $i \in [\ell']$, we have 
    $$|V_I| \geq \left(\frac sr  - \frac{\gamma}{4}\right)(1- 2 \eps)n \geq \left(\frac sr  - \frac{\gamma}{3}\right)n.$$ If $V_I'$ is an arbitrary vertex set of size $sn/r$ containing $V_I$, we have $e_G(V_I') \leq e_G(V_I) + \gamma n^2/3 <\gamma n^2$, which contradicts the assumption on $G$.
	
	Hence, applying Lemma~\ref{l:6.14chr} to $R$ with $\delta_{\ref{l:6.14chr}} = 2\delta$ and $\gamma_{\ref{l:6.14chr}} = \gamma/4$, we arrive at conclusion~(ii) of the Lemma. Let $\cC$ be the obtained collection of $K_2$-copies and $\Qone$-copies (including a~$(\delta, 2)$-absorber $\absorber$), and we may assume that the vertex set of $\cC$ is $V_1, \ldots, V_{\ell''}$, with $\ell'-\ell'' \leq C$. Redefine $V_0 := V_0 \cup V_{\ell''+1} \cup \dots \cup V_{\ell'}$ to be the new leftover set, with $|V_0| \leq \eps n + 2Cn/ \ell' \leq 2\eps n$. Moreover, let $\cM \subset [\ell'']$ be the set of indices $i$ such that the part $V_i$ belongs to the $M$-part in one of the copies of $\Qone$ in $\cC$. Note that since (by the proof of Lemma~\ref{l:6.13chr}) the absorber $\absorber$ contains at least $\delta^2 \ell'' $ copies of $K_2$, and the $M$-part contains a $\frac{t}{s+t}$-proportion of vertices of each copy of $\Qone$, we have 
	\begin{equation}
		\label{eq:tau-estimate}
		|\cM| \leq \frac{t}{s+t}\left(1-2 \delta^2\right)\ell''. 
	\end{equation}
	Hence, for each vertex $x \in V(G)$,
    \begin{equation}
        \label{eq:N(x)-estimate}
        \left|N_G(x) \cap \bigcup_{i \in \cM \cup\{0\}}V_i\right|
	   \leq \left(\frac{t}{s+t} - {\delta^2} \right)n + 2 \eps n.
    \end{equation}
	Since $\eps < \delta^3 < \delta^2/100$ and $|N_G(x)| \geq \frac{nt}{s+t}$, we conclude that $x$ has at least $\delta^2 n / 2$ neighbours outside the above-mentioned set $\bigcup_{i \in \cM \cup\{0\}}V_i$.
	
	Recalling that $\cC$ contains a $(\delta, 2)$-absorber, we obtain the additional property that  for each $i \in [\ell'']$, there are at least $\delta^2 \ell''/96$ indices $j \in [\ell'']$ such that the pair $(V_i, V_j)$ is $(\eps, d_1)$-regular.
	
	Now we will turn the collection $\cC$ into an $K_2^*$-packing.
	Namely, each copy of $K_2$ and $\Qone$ has a $K_2^*$-packing with residue 0 (by Lemma~\ref{l:k2-packing}). This yields a $K_2^*$-packing in $R$, which can be described using a collection of vertex sets $(\{\sigma_i, \tau_i\})_{i \in [\ell]}$  of 
    \[\ell = st\ell''\] 
    copies of $K_2^*$, and a function $f: \{\sigma_i: i \in [\ell]\} \cup \{\tau_i : i \in [\ell]\} \to \{V_j: j \in [\ell'']\}$ such that each $f(\sigma_i) f(\tau_i)$ is an edge in $R$; note that we are using $st\ell''$ copies of $K_2^*$ since the total weight of their vertices is exactly $st\ell''\left(\frac{1}{s(s+t)}+\frac{1}{t(s+t)}\right)=\ell''=|V(\cC)|$.
	
	In covering the graph $G$, there will be more flexibility in how the vertices from the `heavier' parts $f(\sigma_i)$ are used, so we note the following additional properties of the packing~$f$. Firstly, each vertex $x \in V(G)$ has at least $\delta^2 n/ 8$ neighbours in $\bigcup_{i \in [\ell]}f(\sigma_i)$; this follows from (\ref{eq:N(x)-estimate}) and
	$$\bigcup_{i \in [\ell]}f(\sigma_i) \subset V(G) \setminus \bigcup_{i \in \cM \cup\{0\}}V_i,$$
	which holds by the construction of $f$ as stated in Lemma~\ref{l:k2-packing} (only $\tau_i$'s are mapped to the $M$-parts). Secondly, using the absorbing property, for each $i \in [\ell'']$, there are at least $\delta^2 \ell/96$ indices $j \in [\ell]$ such that the pair $(V_i, f(\sigma_j))$ is $(\eps, d_1)$-regular. 
	
	The function $f$ will now be used to subdivide the parts $(V_i)_{i\in[\ell'']}$ into parts $(S'_j)_{j\in[\ell]}$ and $(T'_j)_{j\in[\ell]}$ forming a family of regular pairs $\{(S_j', T_j'):j\in[\ell]\}$. Recall that the weight function $w$ of $K_2^*$ assigns weights $w(\sigma_j) = \frac{1}{t(s+t)}$ and $w(\tau_j) = \frac{1}{s(s+t)}.$ For each vertex $x \in \bigcup_{i\in[\ell'']} V_i$, we will choose a part $S_j'$ or $T_j'$ independently at random; to this end, for $j \in [\ell]$, let $S_j'$ be a random subset of $f(\sigma_j)$, each vertex selected with probability $\frac{1}{s(s+t)}$. Similarly, let $T_j'$ be a random subset of $f(\tau_j)$, with each vertex selected with probability $\frac{1}{t(s+t)}$. The fact that $f$ is a $K_2^*$-packing ensures that this selection is a probability measure for each $x \in \bigcup_{i\in[\ell'']} V_i$. We assume that for each $j$,
	\begin{equation}
		\label{eq:st-sizes}
		|S_j'| = \frac{1}{t(s+t)}|V_1|(1 \pm \eps^2) \quad \text{and} \quad
		|T_j'| = \frac{1}{s(s+t)}|V_1|(1 \pm \eps^2),
	\end{equation}
	as, by Chernoff bounds, this happens with probability $e^{-\Omega(n)}$. Similarly, we assume that each vertex $x \in V(G)$ has at least $\delta^2 n / 16$ neighbours in $\bigcup_{j \in [\ell]}S_j'$. (We remark that this property is the only property for which a random partition is actually needed).
	Finally, let $R'$ be the graph on the vertex set $\{S_j': j \in [\ell]\} \cup \{T_j': j \in [\ell]\}$ which corresponds to $(2r^2\eps, (d_1/2)^+)$-regular pairs. Using Lemma~\ref{l:slicing} (the \textit{Slicing Lemma}) and $\frac{1}{s(s+t)} \geq \frac {1}{r^2}$, any pair stemming from an edge in $R$ will now be an edge in $R'$.  In particular, $(S_j', T_j')$ is an edge in $R'$ for any $j\in[\ell]$ (with the usual abuse of notation of using ordered pairs to describe a graph), and any vertex of $R'$ has at least $\delta^2 \ell/96$ neighbours in $\{S_j': j \in [\ell]\}$.
		
	It remains to obtain super-regular pairs, for which we use the standard fact that most vertices in a (large subset of a) regular pair have an expected proportion of neighbours, see Lemma 2.8 from~\cite{abcd22}. By discarding at most $\eps^{3/4} n/\ell$ vertices from each part $S_j'$ or $T_j'$, we obtain subsets $S_j'' \subset S_j'$ and $T_j'' \subset T_j'$ such 
	that the edges of the graph $R'$ now correspond to $(2r^2\eps, (d_1/4)^+, d_1/4)$-super-regular pairs.  Secondly, we pass to subsets $S_j \subset S_j''$ and $T_j \subset T_j''$ with $t|S_j| = s|T_j|$ for $j \in [\ell]$, after discarding another $\eps^{3/4} n/\ell$ vertices from the relevant part $S_j''$ or $T_j''$; this does not interfere with super-regularity since $\eps \ll d_1$. We conclude that
	\begin{itemize}
		\item the edges of the graph $R'$ (now on the vertex set $\bigcup_{j \in [\ell]}(\{S_j\} \cup \{T_j\})$) correspond to $(\eps^{1/2}, (d_1/8)^+, d_1/8)$-super-regular pairs, which in particular implies~\ref{it:S-reduced},
		\item $\left|V(G) \setminus \bigcup_{j \in [\ell]}(S_j \cup T_j) \right| \leq \eps^{1/2} n$,
		\item  
		$t|S_j| = s|T_j|$ and $|S_j \cup T_j| = \frac{n}{\ell}(1\pm \eps^{1/2})$ for for all $j \in [\ell]$, so~\ref{it:ST-sizes} holds.
	\end{itemize}
	Since  each vertex $x \in V(G)$ \textit{lost} at most $\eps^{1/2}n$ neighbours in $\bigcup_{j \in [\ell]} S_j'$ and $\eps^{1/2}<\delta^4$, claim~\ref{it:S-neighbours} is also ensured. This completes the proof, with constants $d = \delta^3 = (d_1/8)^3$ and  $\eps^{1/2}$ in place of $\eps$.
\end{proof}

We remark that ensuring $t|S_j| = s|T_j|$ was not necessary at this point as we will still need to cover the leftover vertices by copies of $K_r$ for which we will use only vertices from $\bigcup_{j\in[\ell]}S_j$, but it will make the \textit{balancing step} in the proof of Lemma~\ref{l:non-extremal} slightly simpler. 

In most copies of $K_r$ in the factor, the deterministic edges will come from a copy of $K_{s,t}$ in $G$, and the remaining edges will come from the random graph $\Gnp$. Occasionally, some copies of $K_{s-1, t+1}$ in $G$ will be also used.

After applying Lemma~\ref{l:st-reduced} to get a collection $(S_i,T_i)_{i\in[\ell]}$, we let $X$ to be the set of leftover vertices, that is vertices lying outside the set $\bigcup_{i\in[\ell]}(S_i\cup T_i)$.
First, using Lemma~\ref{l:x-cover}, we will cover the vertices of $X$ with a partial $K_r$-factor, but in doing so, we have to be careful, because we are not allowed to occupy any vertices from $\bigcup_{i \in [\ell]} T_i$ -- this is the purpose of \ref{it:S-neighbours} in Lemma~\ref{l:st-reduced}. For, once $\frac{|S_i|}{|T_i|}> \frac st$, any $K_r$ factor in $G^p[S_i \cup T_i]$ would have to contain a~$K_{s+1}$-copy in $S_i$, which is unlikely.

Let $(S_i', T_i')$ be the super-regular pairs leftover after covering $X$. It remains to find a~partial $K_r$-factor such that the remaining super-regular pairs satisfy~Lemma~\ref{l:st-reduced}~\ref{it:ST-sizes}. This will also be done in two steps: first we will find a collection $\cK$ of $O(\ell)$ vertex-disjoint $K_r$-copies such that $|S_i' \cup T_i' \setminus V(\cK)|$ is divisible by $s+t$. Once the size of a pair $(S_i' \setminus V(\cK), T_i' \setminus V(\cK))$ is \textit{divisible by $s+t$}, it can be balanced internally by pulling out an appropriate number of $K_{s-1,t+1}$-copies. We now state the necessary probabilistic lemma.

%\nka{Sort out s, s' and notation for Gp}
\begin{lemma}
	\label{l:divb}
	Let  $s$ and $t' \leq s-1$ be integers, and $\nu >0$. Let $A, B$ be disjoint vertex sets of order $n$, and let $G$ be a bipartite graph on $A \dcup B$ containing at least $\nu n^{s-1 +t'}$ copies of $K_{s-1, t'}$.  Let $W$ be an $n$-vertex set (which may intersect $A$ and $B$). For $V = A\cup B \cup W$, $p = Cn^{-2/s}$ and $C$ sufficiently large, the probability that there is no $K_r$-copy in $G \cup G(V, p)$ with one vertex in $W$, $s-1$ vertices in $A$ and $t'$ vertices in $B$ is at most $e^{-c_s \nu^2  n^2 p}$, where $c_s$ is a constant depending on $s$.
\end{lemma}

\begin{proof}
	We may assume that $W$ is disjoint from $A$ and $B$, by taking sets $A' = A\setminus W$, $B'=B\setminus W$, $W'=W\setminus (A\cup B)$ and assigning each vertex $v \in W \cap A$ to $W'$ or $A'$ uniformly at random, and similarly each vertex $v \in W \cap B$ to $W'$ or $B'$, and then working with smaller sets $W'$, $A'$ and $B'$. Moreover, we may assume that $t' = s-1$ by adding to $B$ a set $B''$ of $n$ auxiliary vertices which are adjacent to every vertex of $A$.    
	Let $\mathcal S$ denote the set of triples $(w, S_A, S_B)$ with $S_A \subset A$, $S_B \subset B$, $|S_A| = |S_B| = s-1$, and such that $(S_A,S_B)$ spans a complete bipartite subgraph in $G$. For a triple $\Sb \in \mathcal{S}$, let $I_\Sb$ be the indicator random variable of the event that $\Sb$ induces a copy of $K_{2s-1}$ in $G \cup G(V,p)$.
	
	Denote 
	$$M = n^{2(s-1)}p^{{2(s-1)}+2\binom{s-1}{2}}
	=  \left(np^{1 + \frac{s-2}{2}} \right)^{2(s-1)} 
	= \left(np^{\frac s2} \right)^{2(s-1)} = C^{s(s-1)}.$$	
	For some constant $c_s'$ depending only on $s$, we have
	\begin{equation}\label{eq:lambda}
        \lambda:= \sum_{\Sb}\er{I_\Sb} \geq c_s' \nu n M.	      
	\end{equation}	
	Now, in order to use Janson's Inequality, let us bound the quantity
	$$\bar \Delta:= \sum_{\Sb} \sum_{ \Sb' \sim \Sb} \er{I_{\Sb}  I_{\Sb'}},$$
	where (as indicated by the relation `$\sim$') the second sum is taken over $\Sb'$ such that $I_\Sb$ and $I_{\Sb'}$ are \textit{not} independent.
	For $\Sb = (w, S_A, S_B)$ and $\Sb '= (w', S_A', S_B')$, let $\ell_A = |S_A \cap S_A'|$ and $\ell_B = |S_B \cap S_B'|$. Define $f(\ell) = n^{-\ell}p^{-\binom{\ell}{2}}$. In case $w \neq w'$, we have
	\begin{equation}
		\label{eq:janson-case1}
		%\er{I_{\Sb} I_{\Sb'}} \leq nM \cdot  n p^{2(s-1)}n^{s-\ell_A} p^{\binom{s-1}{2}- \binom{\ell_A}{2}} n^{s-\ell_B} p^{\binom{s-1}{2}-\ell_B} = n^2 M^2 f(\ell_A) f(\ell_B).
        \sum_{w \neq w'}\er{I_{\Sb} I_{\Sb'}} \leq n^{2+4(s-1)-\ell_A-\ell_B} p^{4(s-1) + 4{s-1\choose 2} - {\ell_A \choose 2} - {\ell_B \choose 2}} \leq n^2 M^2 f(\ell_A) f(\ell_B),
	\end{equation}
	for appropriate values $\ell_A$ and $\ell_B$ to be determined later. Similarly, if $w = w'$, we have
	\begin{align}
		\label{eq:janson-case2}
		%\er{I_{\Sb} I_{\Sb'}} &\leq nM \cdot n^{s-\ell_A} p^{s-1-\ell_A + \binom{s-1}{2}- \binom{\ell_A}{2}} n^{s-\ell_B} p^{s-1 - \ell_B +\binom{s-1}{2}-\binom {\ell_B}{2}} \nonumber \\ 
		%&= nM^2 f(\ell_A)p^{-\ell_A} f(\ell_B)p^{- \ell_B}.
        \sum_{w=w'}\er{I_{\Sb} I_{\Sb'}} &\leq n^{1+4(s-1)-\ell_A-\ell_B}p^{4(s-1) - \ell_A - \ell_B + 4{s-1\choose 2} - {\ell_A \choose 2} - {\ell_B \choose 2}} \leq n M^2 f(\ell_A) p^{-\ell_A} f(\ell_B) p^{-\ell_B},
	\end{align}
	again for appropriate values of $\ell_A$ and $\ell_B$.
	We claim that in both cases, the quantity $\er{I_{\Sb}  I_{\Sb'}}$ is maximised when $\ell_A$ and $\ell_B$ are as small as possible. 
        To see this, note that both functions $f(\ell)$ and $f(\ell)p^{-\ell}$ are log-convex, i.e.,
    $$\frac{f(\ell+1)p^{-\ell - 1}f(\ell - 1)p^{-\ell+1}}{f(\ell)^2 p^{-2\ell}} = \frac{f(\ell+1)f(\ell-1)}{f(\ell)^2} = p^{\frac12 \left(-\ell^2-\ell -\ell^2 + 3\ell -2 + 2\ell^2-2\ell \right)}= p^{-1} >1$$    
    
% ****old*****\\
% To see this for~\eqref{eq:janson-case1}, note that for $\ell \geq 2$ we have
%	$$\frac{f(\ell+1)}{f(\ell)}=n^{-1}p^{-\ell} \geq 1.$$
%	Moreover, since $p = Cn^{-2/s}$ with $C>1$ and $s \geq 2$, we get
%	\[f(0) = 1 > f(1) = n^{-1} > f(2) = n^{-2}p^{-1}.\] 
%***end of old***
   Hence, it suffices to compare the extremal values $f(0)$ and $f(s-1)$, and the values $f(2)$ and $f(s-1)$ as we need $\max\{\ell_A, \ell_B\}\geq 2$. We have
    $$\frac{f(0)}{f(s-1)} = n^{s-1} p^{{s-1 \choose 2}} = C^{{s-1}\choose 2} \geq 1,$$
    and
    $$\frac{f(2)}{f(s-1)} = n^{s-3} p^{{s-1 \choose 2}-1} = C^{{{s-1}\choose 2}-1} \geq 1.$$
	Using these bounds, inequality~\eqref{eq:janson-case1}, and the fact that $\max \{\ell_A, \ell_B\} \geq 2$ whenever $\Sb$ and $\Sb'$ are not independent, we have that in case $w \neq w'$,
	$$\sum_{w\neq w'}\er{I_{\Sb} I_{\Sb'}} \leq n^2M^2 f(2) f(0) = p^{-1}M^2.$$
	The same argument shows that the right-hand side of~\eqref{eq:janson-case2} is minimised when $\ell_A = 1$ and $\ell_B =0$, or vice versa, so when $w = w'$, we have
	$$\sum_{w=w'}\er{I_{\Sb} I_{\Sb'}} \leq nM^2 f(1)p^{-1} f(0) =  p^{-1}M^2.$$
	Using~\ref{eq:lambda}, we conclude that $\bar \Delta \leq 2M^2 p^{-1} = O_s(p^{-1} n^{-2} \nu^{-2}\lambda^2)$.
	
	Using Janson's Inequality (\cite[Theorem 2.14]{jlr00}), we conclude that
	$$\pr{\sum_{\Sb \in \mathcal{S}} I_\Sb = 0} \leq e^{-\frac{\lambda^2}{8 \bar \Delta}} \leq e^{-c_s \nu^2 pn^2},$$
	as required.
\end{proof}

We are now ready to prove Lemma~\ref{l:non-extremal} in the non-singular case $s>t$.

\begin{proof}[Proof of Lemma~\ref{l:non-extremal} in case $s>t$]
	We will be working with positive constants satisfying the hierarchy
	\begin{equation*}
		\gamma \gg  d \gg  \nu \gg \eps' \gg \eps.
	\end{equation*}
	Let $d = d(\gamma)$ be the constant from Lemma~\ref{l:st-reduced}. Recall that $r = s+t$ with $s>t$.
	
	Applying Lemma~\ref{l:st-reduced} (with sufficiently small $\eps$ which is fixed so that the following steps of the proof go through, and $d$ which does not depend on $\eps$), we obtain the sets  $(S_i, T_i)_{i \in [\ell]}$ which form $(\eps, d', d)$-super-regular pairs with some $d' \geq d$. Let $\cS = \{S_i: i \in [\ell]\}$, and $V(\cS) = \bigcup_{i \in [\ell]} S_i$, and define $\cT$ and  $V(\cT)$ analogously for the sets $T_i$. Let $X = V(G) \setminus \bigcup_{i \in [\ell]}(S_i \cup T_i)$ with $|X|\leq \eps n$. Note that $|V(\cS)| \geq \frac 12 (|V(\cS)|+|V(\cT)|)\geq \frac{n}3$. Let $n_*$ be such that all parts $S_i$ and $T_i$ contain at least $n_*$ and at most $rn_*$ vertices; note that $n_* \geq n/(2r \ell)$.
	
	Instead of sampling $\Gnp$ with $p = C n^{-2/s}(\log n)^{2/(s(s-1))}$ at once, we will expose three independent graphs $G_i \sim G(n, p/4)$ for $i = 1,2,3$. The union $G_1 \cup G_2 \cup G_3$ can be viewed as a subgraph of $\Gnp$.
	
	\subsubsection*{Covering the vertices in $X$.} The graph $G_1$ will be used to cover $X$. To retain the minimum degree in the remaining pairs, we will use only a small subset of $V(\cS)$. To this end, let $\vartheta = d/4$ and $\eps \ll \eps' \ll d$, and let $Z$ be a random subset of $V(\cS)$ in which each vertex is included independently with probability $ \vartheta$. Using Chernoff's inequality and regularity inheritance (see Lemma~\ref{l:slicing}), we may assume that $\frac{|Z \cap S_i|}{|S_i|} \in[\theta/2, 3\vartheta /2]$ for $i \in [\ell]$. In particular
    \[ \frac{d}{8} |V(\cS)| \leq |Z| \leq \frac{3d}{8} |V(\cS)|, \]
    and \ref{it:S-neighbours} and~\ref{it:S-reduced} from Lemma~\ref{l:st-reduced} hold for $S_i$ replaced by $S_i'':=Z \cap S_i$, $d$ replaced by $d/2$ and $\eps$ replaced by  $ 4\eps/ \vartheta < \eps'$, that is the following holds.
    \begin{itemize}
        \item[(ii)] Every vertex $x$ has at least $d|Z|/2$ neighbours in $Z=\bigcup_{i\in[\ell]} S_i''$.
        \item[(iii)] For every set $S_i''$, there are at least $d\ell$ sets $S'$ in $\{S''_1,\dots,S''_\ell\}$ such that $(S''_i,S')$ is an $(4\eps/\theta,d',d/2)$-super-regular pair with density $d' \geq d/2$.
    \end{itemize}
    Note also that 
    \[|X| \leq \eps n \leq \frac{6\eps|Z|}{\vartheta} \leq \eps' |Z|.\]
	
	We will next use Lemma~\ref{l:x-cover} with $m=1$ (note that here the graph $B_{m,s,t}$ is just $K_{s,t}$). To verify the hypothesis of the lemma, for a vertex $x \in X$ and a set $Y\subset Z$ with $|Y| \leq 2r|X|$, we will find a suitably large collection of copies of $K_{s,t}$ in $(G[Z]\setminus Y) \cup \{x\}$ containing $x$ in a class of size $t$. %Firstly, the number of parts $S_i''$ with $|S_i'' \cap Y| > \frac{|S_i''|}{2} > \frac{\vartheta \nst}{4}$ is at most
    %\[\frac{4|Y|}{\vartheta \nst} \leq \frac{8r\eps n}{\theta n_*} \leq \frac{16r^2\ell \eps }{\vartheta} \leq r\eps' |Z|.\]
    %Hence, by discarding those parts and adding them to $Y'$, we may pass to a superset $Y' \supset Y$ with $|Y'| \leq 3 r \eps' |Z|$.
    Firstly, we discard the parts $S_i''$ for which $|S_i'' \cap Y| > \frac{|S_i''|}{2}$ and add them to $Y$ thus passing to a superset $Y' \supset Y$. Note that $|Y'| \leq 2|Y| \leq 4r\eps'|Z|$.
	
	Using~\ref{it:S-neighbours}, $x$ has at least $\frac{d|Z|}{2} -|Y'| \geq \frac{d|Z|}{5} $ neighbours in $Z\setminus Y'$, so there is a part $S_i''$ such that $x$ has at least $\frac{d |S_i''|}{5}$ neighbours in $S_i'' \setminus Y'$. Let $N_x$ denote the neighbourhood of $x$ in $S_i''\setminus Y'$. For the part $S_i''$, there is $S_j''$ such that $(S_i''\setminus Y', S_j'' \setminus Y')$ is an $(\eps', d/2)$-regular pair (noticing that here we do not even need super-regularity). Using the Counting Lemma~\cite[Lemma 2.4]{cf13}, we can find at least 
    \[\frac 12 \left(\frac d4 \right)^{st} |N_x|^{s-1}|S_j''|^t \geq \frac 12 \left(\frac{d}{4} \right)^{st} \cdot \left(\frac{d}{5} \right)^{s-1}\left(\frac{\nu}2|S_j|\right)^{s+t-1}\]
    copies of $K_{s-1,t}$ in $G[N_x, S_j'']$, which is at least $\nu |Z|^{s+t-1}$ for some $\nu = \nu(d,r)$. 
	
	Hence we can apply Lemma~\ref{l:x-cover} to obtain a partial $K_r$-factor in $(G \cup G_1)[X \cup Z]$ covering $X$. Let $X^+$  be the vertex set of this partial $K_r$-factor, and recall that $X^+$ is disjoint from $V(\cT)$. Let $S_i' = S_i \setminus X^+$ and $T_i' =T_i$ for each $i\in[\ell]$. Since $|S_i \setminus X^+| \geq |S_i| - |S_i''| \geq (1-3\vartheta/2) |S_i| = (1 - 3d/8)|S_i|$, each vertex from $T_i'$ has at least $5d|S_i'|/8$ neighbours in $S_i'$.

\subsubsection*{Balancing the parts} In the second round, we will \textit{balance} the remaining regular pairs.
Let $G_2 \sim G(n, p/4)$, and assume that $G_2$ has the following properties.
	\begin{description}
		\item[(P1)] For all $i$ and $j \in [\ell]$, if $A \subset S_i$, $B \subset T_i$, and $W \subset S_j$ are subsets of size at least $\nst/2$, then there is a copy of $K_r$ in $G \cup G_2$ with one vertex in $W$, $(s-1)$ vertices in $A$ and $t$ vertices in $B$. %\label{P1}
		\item[(P2)] Let $s', t' \leq s$. For all $i$ and $j \in [\ell]$, if $A \subset S_i$ and $B \subset T_i$ are subsets of size at least $\nst/2$, then there is a copy of $K_{s'+t'}$ in $G \cup G_2$ with $s'$ vertices in $A$ and $t'$ vertices in $B$.% \label{P2}
	\end{description}
	Properties (P1) and (P2) hold with high probability by Lemma~\ref{l:divb}; namely, we can take a~union bound over all choices of $A$, $B$, $W$, and all indices $i, j$, seeing as the bound given by Lemma~\ref{l:divb} is exponential in $C'n^2p$.

Recalling that  $S_i' = S_i \setminus X^+$ and $T_i' = T_i$, we consider the quantity $s|T_i'|-t|S_i'|$, which we call the \textit{defect} of the pair $(S_i', T_i')$. Define $r_i \in \{0, \ldots, r-1\}$ so that 
\[|S_i| - |S_i'| \equiv r_i \pmod{r}.\] 
Note that since $|S_i| - |S_i'| = (|S_i| + |T_i|) - (|S_i'| + |T_i'|)$ and $r$ divides $|S_i|+|T_i|$, we get
\[ |S_i'| + |T_i'| \equiv -r_i \pmod{r}.\] 
Since $r_i \leq |S_i| - |S_i'| \leq \frac{3 \vartheta |S_i|}{2}$ and $|T_i|=|T_i'|$, we can bound the defect $s|T_i'|-t|S_i'|$ as follows
\begin{equation}\label{eq:defect1}
    tr_i = s|T_i|-t |S_i| + tr_i \leq s|T_i'| - t|S_i'| \leq s|T_i|-t|S_i|+\frac{3t\vartheta}{2}|S_i| =\frac{3t\vartheta}{2}|S_i|.
\end{equation}

We will first obtain pairs with a total number of vertices divisible by $r=s+t$. %, so define $a_1, \ldots, a_\ell \in \{0, \ldots, r-1\}$ so that $r \mid |S_i'| + |T_i'|+ a_i$.
%Notice that $\sum_{i \in [\ell]}a_i  \equiv \sum_{i \in \ell} |S_i| - |S_i'| \equiv |X^+|-|X| \equiv 0 \pmod{r}$ since $r \mid n$ and $r \mid X^+$. 
Using (P1) with $W=S_1'$, for $i = 2, \ldots, \ell$ we repeatedly remove $r_i$ copies of $K_r$, each with $s-1$ vertices in $S_i'$, $t$ vertices in $T_i'$, and one vertex in $S_1'$. The property (P1) can be repeatedly applied in this process, since we have only removed at most $r\ell$ vertices in total. Let $S_i^*$ and $T_i^*$ be the remaining parts. %From $S_1'$, we have removed $a_2 + \dots + a_\ell \equiv -a_i \pmod{r}$ vertices, so  $$|S_1^*| + |T_1^*| \equiv |S_i'| - (-a_1) + |T_i'| \equiv 0 \pmod{r}.$$
For $i \in \{2, \ldots, \ell\}$, 
\[|S_i^*| + |T_i^*| \equiv |S_i'| + |T_i'| - r_i(r-1) \equiv 0 \pmod{r}.\] 
Since $r$ divides $\sum_{i \in [\ell]} |S_i^*|+|T_i^*|$, it follows that $r$ also divides $|S_1^*|+|T_1^*|$.

Now, to control the defects, using \eqref{eq:defect1}, for $i \in \{2, \ldots, \ell\}$ we have
$$s|T_i^*|-t|S_i^*| = s|T_i'|-t|S_i'|- s r_it + tr_i(s-1) \geq tr_i - tr_i = 0.$$
For $i=1$ this inequality is clear since $T_1^* = T_1'$.
Moreover, since we removed at most $r\ell$ vertices, it holds that
$$s|T_i^*| - t|S_i^*| \leq 2t \vartheta |S_i|.$$

Finally, for $i \in [\ell]$ we will balance the pairs $(S_i^*, T_i^*)$. To this end, note that $s|T_i^*| - t|S_i^*| \equiv s|T_i^*|+s|S_i^*| \equiv 0 \pmod{r}$, so there is a positive integer $b \leq 2t \vartheta |S_i|/r$ such that
$$s|T_i^*| - t|S_i^*|-b(s+t) = 0.$$
Using (P2), we repeatedly remove $b$ copies of $K_{r}$ with $s-1$ vertices in $S_i$ and $t+1$ vertices in $T_i$; this is possible since $rb \leq \frac{|S_i|}{4}$. The resulting defect is
$$s(|T_i^*|-(t+1)b) - t(|S_i^*|-(s-1)b) = s|T_i^*|-t|S_i^*| - b(s+t) =0.$$

Denote the remaining parts by $(S_i^\dagger, T_i^\dagger)$. They are $(\eps/2, (d/2)^+)$-regular by inheritance. Each vertex $v \in |S_i^\dagger|$ has at least $d|T_i|-2\vartheta|T_i| \geq d|T_i^\dagger|/2$ neighbours in $T_i^\dagger$, and vice versa, so the pair $(S_i^\dagger, T_i^\dagger)$ is $(\eps/2, d'', d/2)$-super-regular for some $d'' \geq d/2$.

Finally, we can expose $G_3$ and apply Lemma~\ref{c:st-reg-pairs-factors} to obtain a $K_r$-factor in $G \cup G_1 \cup G_2 \cup G_3$.
\end{proof}

\subsection{The singular case $s=t$}

We first state a result corresponding to Lemma~\ref{l:st-reduced} in the singular case $(s=t=r/2)$. The proof is omitted as it follows from the same proof of Lemma~\ref{l:st-reduced} with slightly different notation. The reason for the differing notation is that in the singular case we need a collection of edges \textit{and triangles} which cover almost all vertices of the reduced graph. Moreover, the corresponding proof %of Lemma~\ref{l:st-reduced-singular}  
is simpler since there is no need to refine the original partition obtained from the Regularity Lemma (in other words we don't need (i) in Lemma~\ref{l:st-reduced}).

\begin{lemma}
	\label{l:st-reduced-singular}
	Let $s$ be a positive integer,  $1/s \gg \gamma \gg d>0$, and let $\eps>0$ be given. There is $\bar L = \bar L(\gamma, \eps)$ such that the following holds for sufficiently large $n$. Let $G$ be an $n$-vertex graph with $\delta(G) \geq n/2$ in which every vertex subset of size $n/2$ contains at least $\gamma^2 n^2$ edges. There is a collection of parts $S_i$ indexed by $i \in L_2 \cup L_3$, $|L_2 \cup L_3| \leq \bar L$, with the following properties. Let $R$ be the $(\eps, d, d)$-reduced graph on $\{S_i: i \in L_2 \cup L_3\}$.
	\begin{enumerate} 
		\item \label{it-ST-sizes-sing} For every $i,j \in L_2 \cup L_3$,  $|S_i| = |S_j|$ and $\left| V(G) \setminus \bigcup_{i \in L_2 \cup L_3} S_i \right| \leq \eps n$.
        \item \label{it:st-reduced-graph-mindeg} The reduced graph $R$ has minimum degree at least $\left(\frac 12 - 10d^{1/4} \right)|V(R)|$ and no independent set of size $\left(\frac 12 - 20d^{1/4} \right)|V(R)|$.
		\item The reduced graph $R$ contains a perfect matching on $L_2$ and a triangle factor on $L_3$. \label{it:st-reduced-graph-absorb}
%		\item \label{it:S-reduced-sing}  For every set $S_i$, there are at least $d\ell$ $L_3$triangles $A$ in $\absorber$ such that for at least two sets in $A$, $(S_i, S')$ is an $(\eps,d^+, d)$-super-regular pair with density $d' \geq d$. \nka{ugh}
        \item\label{it:st-super-regular} For every set $S_i$ with $i \in L_2 \cup L_3$, there are at least $d|L_2 \cup L_3|$ sets $S_j$ with $j\in L_3$ such that $(S_i, S_j)$ is an $(\eps,d, d)$-super-regular pair.
	\end{enumerate}
\end{lemma}

Next, we will need a lemma which will ensure \textit{short} paths (or trails) in the reduced graph. Recall that a \emph{trail} in a graph $G$ is a sequence $v_0, \ldots, v_\ell$ of vertices of $G$ such that $v_{i-1}v_{i}\in G$ for $i \in [\ell]$ (but the vertices do not have to be distinct). As usual, a~path is a~trail with no repeated vertices. The \emph{length} of a trail $v_0, \ldots, v_\ell$ is $\ell$, that is, the number of edges of the trail. Moreover, this trail is referred to as a $v_0$--$v_\ell$ trail and is called \emph{even} if $\ell$ is even. 

\begin{lemma}
    \label{l:diameter}
    Let $G$ be a connected $n$-vertex graph with $\delta(G)> n/3$, and suppose that each vertex of $G$ belongs to a triangle. Then for any $u, v \in V(G)$, there is an even trail of length at most 8 between $u$ and $v$.
\end{lemma}
    
\begin{proof}
    Given $u$ and $v$, let $P$ be the shortest $u$--$v$ path in $G$, and assume that its length is at least $6$. Then $P$ contains a  vertex $w \in V(P)$ which is at distance at least 3 from both  $u$ and $v$. The neighbourhoods of $u, v$ and $w$ are mutually disjoint, as otherwise there would be a path shorter than $P$ between $u$ and $v$. However, this cannot happen as $\delta(G)> n/3$. Hence $P$ is a $u$--$v$ path of length at most 5.

    If $P$ is even, we are done. If $P$ is not even, we can append a triangle containing $v$ to it, creating a trail of length at most 8.
\end{proof}

Finally, we need a separate statement to cover graphs $G$ which approximately consist of two disjoint cliques of size $n/2$. (To see why this cannot be covered by our usual proof strategy, notice that e.g.\,when $G$ consists of two cliques differing in a small number of vertices, the random edges have to be used to find complete bipartite graphs $K_{s,s}$ between the two cliques.) We defer the proof of this lemma to the end of the section. 

\begin{lemma}
    \label{l:disconnected}
    There is $\xi=\xi(s)$ such that the following holds for sufficiently large $C$. Let $G$ be a graph with $\delta(G) \geq n/2$, and suppose that the vertices of $G$ can be partitioned into sets $A$ and $B$ such that there are at most $\xi n^2$ edges between $A$ and $B$. Then, for $p = Cn^{-2/s}(\log n)^{2/(s(s-1))}$, with high probability, $G \cup \Gnp$ contains a $K_{2s}$-factor.
\end{lemma}

The proof of Lemma~\ref{l:non-extremal} in the singular case follows the same outline as in the non-singular case. Specifically, finding a factor in the reduced graph and covering the remainder $X$ is the same in the two proofs, whereas the \textit{balancing step} is rather different.

\begin{proof}[Proof of Lemma~\ref{l:non-extremal} in case $s=t$]
We will be working with positive constants satisfying the hierarchy
\begin{equation*}
	\gamma \gg  d \gg  \nu \gg \eps.
\end{equation*}
Let $d = d(\gamma)$ be the constant from Lemma~\ref{l:st-reduced-singular}. Recall that $r=2s$.
	
Applying Lemma~\ref{l:st-reduced-singular} (with sufficiently small $\eps$ which is fixed so that the following steps of the proof go through), we obtain  the sets $\{S_i: i \in L_2 \cup L_3 \}$ satisfying the conclusion of the Lemma.
Let $\ell = |L_2 \cup L_3|$ and $X = V(G) \setminus  \bigcup_{i \in L_2 \cup L_3} S_i $ with $|X|\leq \eps n$. Let $n_*$ be the size of each part $S_i$.

Firstly, assume that there is a set $L' \subset L_2 \cup L_3$ with $|L'| \leq \frac{50s \eps}{d}\cdot \ell$ such that the graph $R[L_2 \cup L_3 \setminus L']$ is not connected. Then we can show that the original graph $G$ has a vertex partition $A\cup B$ with $e_G(A, B) = O(\frac{s \eps}{d} n^2)$. Indeed, let $A = \bigcup_{i\in L'} S_i$ and $B = V(G)\setminus A$. Then
\begin{align*} 
    e_G(A,B) & \leq |L'|\nst \cdot (\ell\nst + \eps n) \leq \frac{50s\eps}{d} \ell^2 \nst^2 + \frac{50s\eps}{d} \eps\ell \nst n \leq \frac{100s\eps}{d} n^2.
\end{align*}
In this case, Lemma~\ref{l:disconnected} yields the desired $K_r$-factor.

Hence we assume that after removing any $\frac{50s\eps}{d}\ell$ parts, the remaining $(\eps, d, d)$-reduced graph $R$ is still connected. 
We also remark that, even after removing any $\frac{50 s \eps}{d}\cdot \ell$ parts, the graph $R$ satisfies the hypotheses of Lemma~\ref{l:diameter}. Indeed, the fact that every vertex of $R$ is still in a triangle follows from property~\ref{it:st-reduced-graph-mindeg} of Lemma~\ref{l:st-reduced-singular}.
This also implies that for any $S_i$ and $S_j$ in $L_2 \cup L_3$, there is an even trail of length at most 8.
 
As in the previous proof, we let $p = C n^{-2/s}(\log n)^{-2/(s(s-1))}$ and expose three independent graphs $G_i \sim G(n, p/4)$ for $i = 1,2,3$. The union $G_1 \cup G_2 \cup G_3$ can be viewed as a subgraph of $\Gnp$.
	
\subsubsection*{Covering the vertices in $X$.} The graph $G_1$ will be used to cover $X$. As in the previous proof, we let $Z$ to be a random subset which contains each element of $V(G) \setminus X$ independently with probability $\vartheta = \frac d4$. We assume that $|Z \cap S_i| \leq 3 \vartheta |S_i| /2$ for all $i$. By the previous argument, $Z$ can be chosen so that there is a collection of at most $|X|$ vertex-disjoint $K_r$-copies in $Z \cup X$ covering $X$. Let $X^+$ be the vertex set of this partial $K_r$-factor, and denote  $S_i' = S_i \setminus X^+$ for $i \in L_2 \cup L_3$. We have
\begin{equation}
    |X^+|\leq 2s\eps n \quad \text{and} \quad |S_i'|\geq \left(1-\frac{3\vartheta}{2}\right)n_*.
\end{equation}

\subsubsection*{Ensuring divisibility by $s$} In the first balancing round, we will ensure that the remaining parts have order divisible by $s$ by removing at most $2s\ell$ copies of $K_{r}$. This will be done using a sequence of \textit{alterations}, which we introduce later.
 
Let $G_2 \sim G(n, p/4)$, and assume that $G_2$ has the following property (which typically occurs, and is shown immediately hereafter). 
\begin{description}
%    \item[(P3)] For $i \in [s]$, let $U_i \subset V(G)$ be vertex sets such that $(U_1, U_2)$ and $(U_2, U_3)$ are $(\eps, d^+)$-regular pairs in $G$. If $B_i \cup U_i$ satisfy $|B_0| \geq \frac (d/4)^s |U_0|$ and $|B_i| \geq |U_i|/2$ for $i \in [s]$, then there is a copy of $K_s$ in $G_2$ intersecting each of $B_1, \ldots, B_s$ in one vertex, with $N\left(\tilde K_s\right) \geq (d/4)^s|U_0|$
\item[(P3)] Consider any three parts  $S_{i_1}, S_{i_2}, S_{i_3}$ such that $\{S_{i_1}, S_{i_2}\}$ and $\{S_{i_2}, S_{i_3}\}$ are edges in $R$. If $B_j \subset S_{i_j}$ satisfy $|B_j| \geq |S_{i_j}|/2$ for $j \in [3]$, then there is a copy of $K_{r}$ in $G \cup G_2$ with one vertex in $B_1$, $s$ vertices in $B_2$ and $s-1$ vertices in $B_3$.
\end{description}

This property follows from Lemma~\ref{l:regular-pairs-factors}, although that argument is of course an overkill. Instead, one can use Janson's inequality to find suitable $K_s$-copies in $G_2$ and combine them using regularity. We remark that the property can also be used with $S_{i_1} = S_{i_3}$, yielding a~copy of $K_{r}$ intersecting $S_{i_1}$ and $S_{i_2}$ in exactly $s$ vertices.

Given a set $Y \subset V(G)$, we call a part $S_i$ (or its index $i \in L_2 \cup L_3$) \emph{available} (in $G-Y$) if $|S_i \cap Y | \leq \vartheta \nst$, with $\vartheta = d/4$. Otherwise we call it \emph{unavailable}.

\begin{claim}
    \label{claim:available}
    Let $Y \subset V(G)$ with $|Y|\leq 10s\eps n$. In the graph $G-Y$, there are at most $\frac{50s\eps }{ d}\cdot \ell$ unavailable parts.
        % $\frac{2|Y|}{ \vartheta n}\cdot \ell$ unavailable indices.
\end{claim}
    
\begin{proof}
    If the number of unavailable parts is $\alpha \ell$, then the number of removed vertices $|Y|$ is at least
    $$ |Y| \geq \alpha \ell \vartheta \nst \geq \alpha \vartheta \cdot n(1-\eps).$$

    Rearranging and using $\vartheta = d/4$ and $\eps \ll d$, we get
    $$\alpha \ell \leq \frac{|Y|}{\vartheta n(1-\eps)}\cdot \ell \leq \frac{40 s \eps}{d(1-\eps)}\cdot \ell \leq \frac{50s\eps}{d}\cdot\ell .$$
    % \leq \frac{d\ell}{2}$$ as required.
    %    Now, $S_i$ and $S_j$ have $d$ common neighbours in $R$, so at least one of them is still available.
\end{proof}

    It follows that if $|Y|\leq 10s\eps n$, then for any $S_i$ and $S_j$ in $L_2 \cup L_3$, there is an even trail in $R$ of length at most 8 connecting $S_i$ and $S_j$ and consisting only of parts which are available in $G-Y$.
    
    Next, given $Y\subset V(G)$, \emph{performing an $(i,j,k)$-removal} means removing a copy of $K_{r}$ which uses one vertex from $S_i\setminus Y$, $s$ vertices from $S_j\setminus Y$ and $s-1$ vertices from $S_k\setminus Y$, and adding those $2s$ vertices to $Y$. Furthermore, an \emph{$(i,k)$-alteration} is a move of the following form. Given parts $S_i$ and $S_k$, choose an even trail $P = S_{i_0}S_{i_1}\ldots S_{i_{|P|}}$ with $i_0 = i$, $i_{|P|}=k$, and $|P| \leq 8$. Then perform a sequence of removals $(i_0, i_1, i_2)$, $(i_2, i_3, i_4)$, ..., $(i_{|P-2|}, i_{|P-1|}, i_{|P|})$. The sizes of the sets $S_i$ and $S_k$ change by $-1$ and $+1$ modulo $s$, respectively, whereas all other parts remain intact modulo $s$.
    
    This move will now be repeatedly used to ensure divisibility by $s$. To this end, set $Y=X^+$, and for $i \in L_2 \cup L_3$, let $r_i$ be an integer such that $|S_i'| \equiv r_i \pmod{s}$, $|r_i| < s$ and $\sum_{i \in L_2 \cup L_3} r_i = 0$. (Note that such a choice of $r_i$'s exists since the total number of vertices is divisible by $2s$ throughout the removal process). Let $r=\max_i |r_i|$ and choose $r$ auxiliary matchings $\pi_1, \dots, \pi_r$ on $L_2 \cup L_3$ (regardless of the reduced graph $R$) such that for $\pi=\pi_1\cup \ldots\cup \pi_r$ each edge $ik$ of $\pi$ satisfies $r_i>0$ and $r_k<0$ (or $r_i<0$ and $r_k>0$), and each index $i$ has degree $|r_i|$ in $\pi$. Next, for each edge $ik$ in $\pi$ (in any order), perform an $(i,k)$-alteration. Note that this is possible since at most $s\ell/2$ alterations are performed, removing at most $4s^2 \ell$ vertices in total. Hence, using Claim~\ref{claim:available}, between any $S_i$ and $S_k$ we have a sufficient trail. Let $S_i''\subset S_i$ be the resulting parts. Then for each $i\in L_2\cup L_3$ we have $|S_i''| \equiv 0 \pmod{s}$. Moreover
    $$ \quad |S_i''| \geq \left(1 - 7\vartheta/4\right) n_* \quad \text{and} \quad |Y| \leq 5s\eps n/2.$$

\subsubsection*{Ensuring divisibility by $2s$} 
    Let $L^*$ be the set of parts whose size is \textit{not} divisible by $2s$ (so it is congruent to $s$ modulo $2s$). Choose an auxiliary perfect matching $\rho$ on $L^*$. Now, for each edge $ik$ in the matching $\rho$, perform the following move. Choose a path $P$ between $S_i$ and $S_k$ of length at most $8$ in $R$, and for each edge $S_{j}S_{j'}$ in $P$, perform a~$(j,j',j)$-removal. After this procedure, the size of $S_i''$ and $S_k''$ change by $s$, and the sizes of the remaining parts remain intact modulo $2s$.

    Since $\rho$ has at most $\ell/2$ edges, we removed at most $8 \cdot 2s \cdot \ell/2 \leq 8s \ell$ vertices in this round. Denote the remaining subsets by $S_i^* \subset S_i''$. Then for each $i\in L_2\cup L_3$ we have $|S_i^*| \equiv 0 \pmod{2s}$. Moreover
    $$ \quad |S_i^*| \geq \left(1 - 15\vartheta/8\right) n_* \quad \text{and} \quad |Y| \leq 3s\eps n.$$

    \subsubsection*{Balancing the `regular pairs'}
    At this point we need some additional notation. Let $M$ be a~perfect matching on $L_2$ and let $T$ be a triangle-factor on $L_3$ granted by Lemma~\ref{l:st-reduced-singular}~(iii).

    We will now pass to subsets $S_i^\dagger \subset S_i^*$ for $i \in L_2$ satisfying $|S_i^\dagger| = |S_{j}^\dagger| \equiv 0 \pmod{s}$ for $ij \in M$. Note that 
    %for each $ij$, $|S_i^*| - |S_{j}^*| \equiv 0 \pmod{2s}$, and, 
    since we already removed in total at most $3s\eps n$ vertices from the original parts,
    \begin{equation}
        \label{eq:diff-singular}
        \sum_{ij \in M}||S_i^*|-|S_j^*|| \leq 3s \eps n.
    \end{equation}

    Now take an arbitrary edge $ij$ of $M$, and assume that $|S_i^*|>|S_j^*|$. Note that
    $|S_i^*| - |S_{j}^*| \equiv 0 \pmod{2s}$ and $|S_i^*| - |S_{j}^*| \leq 2\vartheta \nst$.
    As long as the leftover of $S_i^*$ is larger that $S_j^*$, we repeat the following move.
    Choose an available part $S_k$ with $k \in L_3$ which is adjacent to $S_i$ in $R$ -- such a part exists since (by property~\ref{it:st-super-regular} of Lemma~\ref{l:st-reduced-singular}) $S_i^*$ has at least $d \ell$ neighbours in $L_3$ and (using Claim~\ref{claim:available}) the number of unavailable parts is at most $\frac{50s\eps}{d}\cdot\ell$. Perform two $(i,k,i)$-removals.
    
    The resulting sets have sizes divisible by $2s$, and we have performed at most $3 \eps n$ removals, losing at most $6s \eps n$ vertices.
    
    Let $G^\dagger$ be the leftover graph with parts $S_i^\dagger \subset S_i^*$, and note that as claimed, $|V(G)\setminus V\left(G^\dagger \right)| \leq 9s\eps n$. Moreover, at most $2\vartheta\nst = d\nst/2$ vertices are removed from each part, so the remaining pairs and triples corresponding to $M$ and $T$, respectively, are $(2\eps, (d/2)^+, d/2)$-super-regular.

    It remains to split each regular triple $(S_i^\dagger, S_j^\dagger, S_k^\dagger)$ with $ijk \in T$ into three balanced pairs, as follows. Fix a triple $(S_1^\dagger, S_2^\dagger, S_3^\dagger)$, and let $x_1, x_2$ and $x_3$ satisfy
    \begin{align*}
        x_1 + x_2  &= |S_{1}^{\dagger}|\,, \\
        x_2 + x_3  &= |S_{2}^{\dagger}|\,, \\
        x_3 + x_1  &= |S_{3}^{\dagger}|\,.
    \end{align*}
    Since $x_1 = \frac 12 \left( |S_{1}^{\dagger}| + |S_{3}^{\dagger}|- |S_{2}^{\dagger}| \right)$, we have $x_1 \geq \frac {\nst}{2}(1-d) \geq \frac{\nst}{3}$, and $x_1$ is an integer divisible by $s$. The same holds for $x_2$ and $x_3$.

    Now we can split $(S_1^{\dagger}, S_{2}^{\dagger}, S_{3}^{\dagger})$ into three pairs $(U_j, W_j)$, $j\in[3]$, with $|U_j| = |W_j| = x_j$, and such that $(U_j, W_j)$ is an $(8\eps, (d/3)^+, (d/8))$-super-regular pair. (Technically, we can take a uniformly random partition and use Chernoff-type bounds for the hypergeometric distribution. See~\cite{skala2013hypergeometric}.) 

    \subsubsection*{Covering the regular pairs}
    We obtain a collection of $(8\eps, (d/3)^+, (d/8))$-super-regular pairs which satisfy the hypothesis of Lemma~\ref{c:st-reg-pairs-factors}. Now we can expose the random graph $G_3$ and apply Lemma~\ref{c:st-reg-pairs-factors} to obtain a $K_r$-factor in $G ^\dagger \cup G_3$. This completes the proof.
\end{proof}

Now we are ready to prove Lemma~\ref{l:disconnected}. We remark that the logarithmic factor in $p$ is probably not needed, but is an artefact of using Lemma~\ref{l:x-cover}.

\begin{proof}
    As before, let $G_1 \sim G_{n, p/3}$ and $G_2 \sim G_{n, p/3}$ be mutually independent random graphs.
    
    Let $X_A$ be the set of vertices in $A$ with more than $\sqrt{\xi}n$ neighbours in $B$, and define $X_B$ similarly. We have $|X_A|\cup |X_B| \leq 2\sqrt{\xi} n$, as otherwise there would be more than $\xi n^2$ edges between $A$ and $B$. Let $X = X_A \cup X_B$.
    
    We claim that $X$ and $G$ satisfy the hypothesis of Lemma~\ref{l:x-cover} with $\eps_{\ref{l:x-cover}} = 2\sqrt{\xi}$ and $\nu$ a constant depending only on $s$. To verify this, let $|Y | \leq 4s\eps n$ and fix $x \in X$. Let $A_x \subset N_G(x)\setminus Y$ be an arbitrary subset of order $n/4$. In the bipartite graph $G[A_x, V(G) \setminus (A_x \cup Y \cup \{ x\})]$, every vertex from $A_x$ has at least $n/8$ neighbours. Hence, using the supersaturation version of the K\H{o}v\'ari--S\'os--Tur\'an Theorem~\cite{kst54}, one may find a collection of $\nu n^{2s-1}$ copies of $K_{s-1, s}$ with the larger class contained in $N_G(x)$. Thus $|\mathcal{B}_x[V(G)\setminus Y]| \geq \nu n^{2s-1}$ and, by Lemma~\ref{l:x-cover}, $X$ can be covered by a partial $K_{r}$-factor in $G \cup G_1$. Let $X^+$ be the vertices contained in this $K_{r}$-factor.
    
    Let $A' = A \setminus X^+$ and $B' = B \setminus X^+$. Notice that since $G[A'\cup B']$ is a graph with minimum degree at least $(1/2-4s\sqrt{\xi}n)$ and each $x\in A'$ has at most $\sqrt{\xi}n$ neighbours in $B'$, we have $|A'| \geq (1/2-(4s+1)\sqrt{\xi})n$, $|B'| \leq (1/2+(4s+1)\sqrt{\xi})n$ and $\delta(G[A']) \geq (1/2-(4s+1)\sqrt{\xi})n$. 
    %$G[A']$ is a graph with minimum degree at least $(1/2 -  \sqrt{\xi} - 2\sqrt{\xi})n = (1/2 - 3\sqrt{\xi})n$, so in particular, $|A'| \geq (1/2 - 3\sqrt{\xi})n$ and $|B'| \leq (1/2 + 3\sqrt{\xi})n$. The same upper bound holds for $|A'|$, so the minimum degree in $G[A']$ satisfies 
    By symmetry, we also have $|A'| \leq (1/2+(4s+1)\sqrt{\xi})n$, so the minimum degree in $G[A']$ satisfies
    \begin{equation}
        \frac{\delta(G[A'])}{|A'| } \geq \frac{1/2-(4s+1)\sqrt{\xi}}{1/2+(4s+1)\sqrt{\xi}} \geq 1-4(4s+1)\sqrt{\xi}.
    \end{equation}
    The same conclusion holds for $B'$.

    The next step is to ensure divisibility by $2s$, using the second random graph $G_2$. By the same argument as for Lemma~\ref{l:ks-in-lin-sized} (Janson's Inequality), one can show that at our value of $p$, $G_2$ has the following property with high probability.

    \begin{description}
        \item[(P)] For any two $A_1 \subset A$ and $B_1 \subset B$ of order at least $n/4$, there is a copy of $K_{r}$ in $G \cup G_2$ intersecting $A_1$ in $s+1$ vertices and $A_2$ in $s-1$ vertices. 
    \end{description}
    
    (We remark that unlike in the previous case, now the edges between $A_1$ and $B_1$ will come from the random graph $G_2$.)
    % for the proof, it's probably better to focus on NON-neighbours in A_1, and show that A_1 has linearly many s+1-cliques. Same for B_1. This gives linearly many opportunities, so

    Hence we assume that $G_2$ satisfy \textbf{(P)}. Since $|A'|+|B'|$ is divisible by $2s$, there is a positive integer $a < 2s$ such that $|A'| \equiv  - |B'| \equiv a \pmod{2s}$. We can repeatedly apply \textbf{(P)} to remove $a$ copies of $K_{r}$ with $s+1$ vertices in $A'$ and $s-1$ vertices in $B'$.
    Let $A^* \subset A'$ be the remaining subset of $A'$. For sufficiently large $n$, the remaining graph $G[A^*]$ has minimum degree at least $(1- 5(4s+1)\sqrt{\xi})|A^*|.$ By the theorem of Hajnal and Szemer\'edi~\cite{Hajnal1970}, assuming $5(4s+1)\sqrt{\xi}< 1/(2s)$, $G[A^*]$ has a $K_{r}$-factor. The same conclusion holds for $G[B^*]$. This completes the proof of the Lemma.
\end{proof}

\section*{Acknowledgements} 
N.K. is supported by the Croatian Science Foundation, project number HRZZ-IP-2022-10-5116 (FANAP).

We would like to thank Hiêp Hàn for initial discussions regarding this project. We are also thankful to Mathias Schacht and the Department of Mathematics at the University of Hamburg for hosting S.A. and N.K.

\bibliographystyle{plain}

\begin{thebibliography}{50}

\bibitem{abcd22}
P.~Allen, J.~B{\"o}ttcher, J.~Corsten, E.~Davies, M.~Jenssen, P.~Morris,
  B.~Roberts, and J.~Skokan.
\newblock A robust {C}orr{\'a}di--{H}ajnal theorem.
\newblock {\em arXiv preprint arXiv:2209.01116}, 2022.

\bibitem{adrrs21}
S.~Antoniuk, A.~Dudek, Chr. Reiher, A.~Ruci\'{n}ski, and M.~Schacht.
\newblock High powers of {H}amiltonian cycles in randomly augmented graphs.
\newblock {\em J. Graph Theory}, 98(2):255--284, 2021.

\bibitem{adr23}
S.~Antoniuk, A.~Dudek, and A.~Ruci\'{n}ski.
\newblock Powers of {H}amiltonian cycles in randomly augmented {D}irac graphs
  -- the complete collection.
\newblock {\em J. Graph Theory}, 104(4):811--835, 2023.

\bibitem{balog16}
J.~Balogh, T.~Molla, and M.~Sharifzadeh.
\newblock Triangle factors of graphs without large independent sets and of
  weighted graphs.
\newblock {\em Random Struct. Algorithms}, 49(4):669--693, 2016.

\bibitem{btw18}
J.~Balogh, A.~Treglown, and A.~Z. Wagner.
\newblock Tilings in {R}andomly {P}erturbed {D}ense {G}raphs.
\newblock {\em Comb. Probab. Comput.}, 28(2):159--176, 2018.

\bibitem{bhkm19}
W.~Bedenknecht, J.~Han, Y.~Kohayakawa, and G.~O. Mota.
\newblock Powers of tight {H}amilton cycles in randomly perturbed hypergraphs.
\newblock {\em Random Structures Algorithms}, 55(4):795--807, 2019.

\bibitem{bfm04}
T.~Bohman, A.~Frieze, M.~Krivelevich, and R.~Martin.
\newblock Adding random edges to dense graphs.
\newblock {\em Random Structures Algorithms}, 24(2):105--117, 2004.

\bibitem{bfm03}
T.~Bohman, A.~Frieze, and R.~Martin.
\newblock How many random edges make a dense graph hamiltonian?
\newblock {\em Random Structures Algorithms}, 22(1):33--42, 2003.

\bibitem{bmpp20}
J.~B\"{o}ttcher, R.~Montgomery, O.~Parczyk, and Y.~Person.
\newblock Embedding spanning bounded degree graphs in randomly perturbed
  graphs.
\newblock {\em Mathematika}, 66(2):422--447, 2020.

\bibitem{bpss23}
J.~B{\"o}ttcher, O.~Parczyk, A.~Sgueglia, and J.~Skokan.
\newblock Triangles in randomly perturbed graphs.
\newblock {\em Combinatorics, Probability and Computing}, 32(1):91--121, 2023.

\bibitem{bpss24}
J.~B{\"o}ttcher, O.~Parczyk, A.~Sgueglia, and J.~Skokan.
\newblock The square of a hamilton cycle in randomly perturbed graphs.
\newblock {\em Random Structures \& Algorithms}, 65(2):342--386, 2024.

\bibitem{cf13}
D.~Conlon and J.~Fox.
\newblock Graph removal lemmas.
\newblock In {\em Surveys in combinatorics 2013}, volume 409 of {\em London
  Math. Soc. Lecture Note Ser.}, pages 1--49. Cambridge Univ. Press, Cambridge,
  2013.

\bibitem{ch63}
K.~Corr\'{a}di and A.~Hajnal.
\newblock On the maximal number of independent circuits in a graph.
\newblock {\em Acta Math. Acad. Sci. Hungar.}, 14:423--439, 1963.

\bibitem{dirac52}
G.~A. Dirac.
\newblock Some theorems on abstract graphs.
\newblock {\em Proc. London Math. Soc. (3)}, 2:69--81, 1952.

\bibitem{fknp21}
K.~Frankston, J.~Kahn, B.~Narayanan, and J.~Park.
\newblock Thresholds versus fractional expectation-thresholds.
\newblock {\em Ann. of Math. (2)}, 194(2):475--495, 2021.

\bibitem{Hajnal1970}
A.~Hajnal and E.~Szemer\'{e}di.
\newblock Proof of a conjecture of {P}. {E}rd{\H{o}}s.
\newblock In {\em Combinatorial theory and its applications, {II} ({P}roc.
  {C}olloq., {B}alatonf\"{u}red, 1969)}, pages 601--623. North-Holland,
  Amsterdam, 1970.

\bibitem{hmt21}
J.~Han, P.~Morris, and A.~Treglown.
\newblock Tilings in randomly perturbed graphs: bridging the gap between
  {H}ajnal--{S}zemer\'{e}di and {J}ohansson--{K}ahn--{V}u.
\newblock {\em Random Structures Algorithms}, 58(3):480--516, 2021.

\bibitem{heckel21}
A.~Heckel.
\newblock Random triangles in random graphs.
\newblock {\em Random Structures Algorithms}, 59(4):616--621, 2021.

\bibitem{hhp19}
J.~Hladk\'{y}, P.~Hu, and D.~Piguet.
\newblock Koml\'{o}s's tiling theorem via graphon covers.
\newblock {\em J. Graph Theory}, 90(1):24--45, 2019.

\bibitem{janson90}
S.~Janson.
\newblock Poisson approximation for large deviations.
\newblock {\em Random Structures Algorithms}, 1(2):221--229, 1990.

\bibitem{jlr00}
S.~Janson, T.~{\L}uczak, and A.~Ruci\'{n}ski.
\newblock {\em Random graphs}.
\newblock Wiley-Interscience Series in Discrete Mathematics and Optimization.
  Wiley-Interscience, New York, 2000.

\bibitem{jkv08}
A.~Johansson, J.~Kahn, and V.~Vu.
\newblock Factors in random graphs.
\newblock {\em Random Structures Algorithms}, 33(1):1--28, 2008.

\bibitem{jk20}
F.~Joos and J.~Kim.
\newblock Spanning trees in randomly perturbed graphs.
\newblock {\em Random Structures Algorithms}, 56(1):169--219, 2020.

\bibitem{kst54}
T.~K\H{o}vari, V.~S\'{o}s, and P.~Tur\'{a}n.
\newblock On a problem of {K}. {Z}arankiewicz.
\newblock {\em Colloq. Math.}, 3:50--57, 1954.

\bibitem{komlos00}
J.~Koml\'{o}s.
\newblock Tiling {T}ur\'{a}n theorems.
\newblock {\em Combinatorica}, 20(2):203--218, 2000.

\bibitem{kss97}
J.~Koml{\'o}s, G.~N. S{\'a}rk{\"o}zy, and E.~Szemer{\'e}di.
\newblock Blow-up lemma.
\newblock {\em Combinatorica}, 17:109--123, 1997.

\bibitem{ks95}
J.~Koml{\'o}s and M.~Simonovits.
\newblock Szemer\'edi's {R}egularity lemma and its applications in graph
  theory, 1995.

\bibitem{kks17}
M.~Krivelevich, M.~Kwan, and B.~Sudakov.
\newblock Bounded-degree spanning trees in randomly perturbed graphs.
\newblock {\em SIAM J. Discrete Math.}, 31(1):155--171, 2017.

\bibitem{kst06}
M.~Krivelevich, B.~Sudakov, and P.~Tetali.
\newblock On smoothed analysis in dense graphs and formulas.
\newblock {\em Random Structures Algorithms}, 29(2):180--193, 2006.

\bibitem{ko09}
D.~K\"{u}hn and D.~Osthus.
\newblock The minimum degree threshold for perfect graph packings.
\newblock {\em Combinatorica}, 29(1):65--107, 2009.

\bibitem{pp22}
J.~Park and H.~T. Pham.
\newblock A proof of the {K}ahn--{K}alai conjecture.
\newblock In {\em 2022 IEEE 63rd Annual Symposium on Foundations of Computer
  Science (FOCS)}, pages 636--639. IEEE, 2022.

\bibitem{perkins2024searching}
W.~Perkins.
\newblock Searching for (sharp) thresholds in random structures: where are we
  now?
\newblock {\em arXiv preprint arXiv:2401.01800}, 2024.

\bibitem{psss22}
H.~T. Pham, A.~Sah, M.~Sawhney, and M.~Simkin.
\newblock A toolkit for robust thresholds.
\newblock {\em arXiv preprint arXiv:2210.03064v3}, 2022.

\bibitem{Posa1976}
L.~P{\'o}sa.
\newblock {H}amiltonian circuits in random graphs.
\newblock {\em Discrete Math.}, 14(4):359--364, 1976.

\bibitem{riordan22}
O.~Riordan.
\newblock Random cliques in random graphs and sharp thresholds for
  {$F$}-factors.
\newblock {\em Random Structures Algorithms}, 61(4):619--637, 2022.

\bibitem{rucinski92}
A.~Ruci\'{n}ski.
\newblock Matching and covering the vertices of a random graph by copies of a
  given graph.
\newblock {\em Discrete Math.}, 105(1-3):185--197, 1992.

\bibitem{sz03}
A.~Shokoufandeh and Y.~Zhao.
\newblock Proof of a tiling conjecture of {K}oml\'{o}s.
\newblock {\em Random Structures Algorithms}, 23(2):180--205, 2003.

\bibitem{skala2013hypergeometric}
M.~Skala.
\newblock Hypergeometric tail inequalities: ending the insanity.
\newblock arXiv:1311.5939, 2013.

\bibitem{st04}
D.~A. Spielman and S.~Teng.
\newblock Smoothed analysis of algorithms: why the simplex algorithm usually
  takes polynomial time.
\newblock {\em J. ACM}, 51(3):385--463, 2004.

\end{thebibliography}

\end{document}